\documentclass[a4,10pt]{article}
\usepackage[a4paper,left=25mm,right=25mm,top=25mm,bottom=30mm]{geometry}
\usepackage{setspace}
\usepackage{amssymb,amsthm}
\usepackage{mathtools}
\usepackage{ascmac}
\usepackage{graphicx}
\usepackage{hyperref}
\usepackage{color}
\usepackage{algpseudocode}
\usepackage{algorithm}
\usepackage{multicol,multirow}
\usepackage{boites}
\usepackage{framed}
\usepackage{setspace}
\usepackage{xcolor}
\hypersetup{
	colorlinks=true,
	citebordercolor=green,
	linkbordercolor=red,
	urlbordercolor=cyan,
}

\theoremstyle{plain}
\newtheorem{theorem}{Theorem}
\newtheorem{prop}[theorem]{Proposition}
\newtheorem{lemma}[theorem]{Lemma}

\theoremstyle{definition}
\newtheorem{definition}[theorem]{Definition}

\theoremstyle{remark}
\newtheorem{rem}[theorem]{Remark}
\newtheorem{example}[theorem]{Example}

\numberwithin{theorem}{section}
\numberwithin{equation}{section}

\newcommand\Z{\mathbb{Z}}

\newcommand\R{\mathbb{R}}

\newcommand\T{\mathbb{T}}
\newcommand\vhi{\varphi}
\newcommand\eps{\varepsilon}

\newcommand\bsym[1]{\boldsymbol{#1}}

\newcommand\del{\partial}
\renewcommand\bar[1]{\overline{#1}}

\newcommand\argmin[2]{\underset{#1}{\mathrm{argmin}}\,{#2}}

\definecolor{deeppink}{HTML}{ff1493}
\definecolor{royalblue}{HTML}{4169e1}
\definecolor{gainsboro}{HTML}{dcdcdc}

\newcommand{\add}[1]{#1} % saito
\title{
\emph{On the Generalized Conditional Gradient Method for Mean Field Games with Local Coupling Terms}
}
\author{
Haruka Nakamura\thanks{Graduate School of Mathematical Sciences, The University of Tokyo, Komaba 3-8-1, Meguro-ku, Tokyo 153-8914, Japan. \texttt{nakamura-haruka@g.ecc.u-tokyo.ac.jp}}
 \quad and \quad 
 Norikazu Saito\thanks{Graduate School of Mathematical Sciences, The University of Tokyo, Komaba 3-8-1, Meguro-ku, Tokyo 153-8914, Japan. \texttt{norikazu@g.ecc.u-tokyo.ac.jp}}
}

\date{\today}

\begin{document}

\maketitle

%%%
%%% abstract
%%% 

\begin{abstract}
We study the generalized conditional gradient (GCG) method for time-dependent second-order mean field games (MFG) with local coupling terms. While explicit convergence rates of the GCG method were previously established only for globally coupled interactions, the assumptions used there fail to cover typical local interactions such as congestion effects. To overcome this limitation, we introduce a refined analytical framework adapted to local couplings and derive explicit convergence estimates in terms of the exploitability and optimality gap. The key difficulty lies in establishing uniform bounds on the Hamilton--Jacobi--Bellman solutions; this is solved via the Cole--Hopf transformation under a standard quadratic Hamiltonian with a convection effect. We further provide numerical experiments demonstrating convergence behavior and confirming the theoretical rates. Additionally, the existence and uniqueness of smooth solutions to the MFG system with locally coupled interactions are established. %
\end{abstract}

\bigskip


\bigskip

\noindent \textbf{AMS classification (2020):} 90C52, 91A16, 91A26, 91B06, 49K20, 35F21, 35Q91.

\section{Introduction}
\label{sec:intro}

The Mean Field Games (MFG) system provides a mathematical framework for modeling optimization problems in which a very large number of agents interact with each other.
It was independently introduced in 2006 by 
Lasry and Lions \cite{ll06a,ll06b,LL07} and by Huang et al. \cite{H06}, as is well known.
The MFG system has been widely applied in fields that study large-scale collective behavior from a macroscopic perspective; see 
\cite{MR3752669,MR3753660,lauriere2024,10508221} for more detail.

This paper is motivated by the work of Lavigne and Pfeiffer \cite{LP23}.
In their paper, they studied the \emph{Generalized Conditional Gradient} (GCG) method for time-dependent second-order MFG and, for the first time, established explicit convergence rates.
Their analysis focuses on \emph{globally} coupled interactions in the spatial variable, and this assumption plays an essential role at several key steps in the proofs of their main theorems.
The purpose of the present paper is, in a nutshell, to extend their results to the case of \emph{locally} coupled interactions. To explain this more concretely, we now introduce our problem. 
Letting $Q := (0, T) \times \T^d$ with the $d$-dimensional torus $\mathbb{T}^d := (\mathbb{R}/\mathbb{Z})^d$ and $T > 0$, we consider the MFG system of the form:  
\begin{subequations} 
\label{MFG_eq}
\begin{alignat}{3}
 -\del_t u - \nu \Delta{u} + H(t, x, \nabla{u}) &= f(t,x,m) &\quad& \mbox{in } Q, \label{MFG_eq1}\\ 
 u(T, x) &= g(x) && \mbox{on } \T^d,\label{MFG_eq2}\\
\del_t m - \nu \Delta{m} - \nabla \cdot (m \nabla_p H(t, x, \nabla u)) &= 0 && \mbox{in } Q, \label{MFG_eq3}\\ 
m(0, x) &= m_0(x) && \mbox{on } \T^d,\label{MFG_eq4}
\end{alignat}
%\tag{MFG}
\end{subequations}
where $\nu > 0$ is a given constant.  The unknown functions are $u:Q\to \mathbb{R}$ and $m:Q\to \mathbb{R}$, which we call the value and density functions, respectively. The symbols \(\Delta\), \(\nabla\), and \(\nabla\!\cdot\) denote the Laplacian, gradient, and divergence operators in \(\mathbb{R}^d\), respectively. 
We are given $g : \T^d \to \R$, 
$m_0 : \T^d \to \R$, and 
$f: Q \times \mathcal{D}_1(\T^d) \to \R$, and they are called the terminal condition, initial condition and coupling term, respectively. The space  $\mathcal{D}_1(\T^d)$ denotes a set of density distributions on $\T^d$ which will be defined in \eqref{eq:d1} below. We interpret the right-hand side of \eqref{MFG_eq1} as \(f(t,\cdot,m(t,\cdot))\), and, when no confusion arises, write it simply as \(f(t,x,m)\) or \(f(\cdot,\cdot,m)\). 
The function $H : Q \times \R^d \to \R$ is called the Hamiltonian and is usually defined as the Legendre transformation of the running cost $L: Q \times \R^d \to \R$, 
\begin{equation}
\label{eq:r-cost}
H(t, x, p) := L^{\ast}(t, x, -p) = \sup_{v \in \R^d} \left[-p\cdot v - L(t, x, v)\right],
\quad (t,x,p)\in Q\times \R^d.
\end{equation}
We assume that $H$ is differentiable with respect to its third variable $p$, and denote its gradient by $\nabla_p H(t,x,p)$. The coupling term $f$ can describe the congestion cost in the dynamics (see Section \ref{sec_numerical_experiment} for the details on the effects). 
At this stage, it is useful to introduce the form of the coupling terms that we aim to study in this paper.
Namely, we consider coupling terms that represent local interactions, as in the example below: \add{
\begin{equation}
\label{eq:ex-f}
f(t, x, m) = C_1\min\{m(t,x)^\alpha,\beta\}
%\varphi(m(t,x)) ,
\end{equation}
where $C_1>0$, $\alpha\ge 1$ and $\beta\ge 1$ are given constants. Here, $m$ is assumed to be non-negative and $\beta$ can be taken sufficiently large.}  
This is only a prototypical example, not an exhaustive list.
The precise assumptions imposed on $f$ considered in our work will be stated in 
\textup{(f-B)} 
\textup{(f-L)} 
\textup{(f-M)}, and \textup{(f-P)} 
below. Such local coupling terms are of significant practical relevance. As a concrete example, Bonnemain et al. \cite{B23} proposed an MFG model with $f(t,x,m)=m(t,x)$ to describe the response of a crowd to an intruder moving through a static group.
They demonstrated that MFG successfully capture rational lateral avoidance behavior of the crowd, which cannot be reproduced by traditional mathematical models.
Furthermore, in many studies on the MFG system, numerical experiments are often performed using local coupling terms; see \cite{ach13, A12, A13, AC10, C21, G12, I23, MR4596353, os25} for example. 
\add{
Beyond the numerical literature, local coupling terms have also played an important role in the theoretical development of MFG theory. For instance, \cite{CP20} treats global and local couplings separately in its discussion of existence and uniqueness results. This distinction is reflected in a substantial body of related work; see, for example, \cite{C15,CGPT15}, which study first-order and degenerate second-order MFG systems with local couplings. These works highlight the fact that local couplings often require analytical techniques that differ significantly from those used in the globally coupled setting.
}

The ultimate goal of our research is to develop efficient and practical numerical methods for the MFG system and to establish their convergence properties. However, introducing a discretization alone is not sufficient for obtaining a practical numerical solver for the MFG system. Thus, unlike initial value problems, it cannot be solved by marching forward in time; instead, the entire space–time domain must be handled simultaneously.
Even in two spatial dimensions, the number of degrees of freedom in the discretized finite-dimensional system becomes extremely large, requiring high-performance computing resources and long computational times.
To address this issue, it is often useful to introduce an iterative-decoupling reformulation before discretization.
More precisely, we decouple the Hamilton--Jacobi--Bellman (HJB) equation \eqref{MFG_eq1}--\eqref{MFG_eq2} and the Fokker–Planck (FP) equation \eqref{MFG_eq3}--\eqref{MFG_eq4}, and solve them alternately in an iterative manner to compute a solution of the MFG system. The policy iteration \cite{C21,CT22} and monotone iteration \cite{G12} are typical examples. 
In this paper, we focus in particular on the following iterative scheme, known as the GCG method: starting with the initial guess $\bar{m}_0$, we generate the sequences ${u}_{k}$, ${m}_{k}$, $\gamma_k$ and $\bar{m}_{k+1}$ for $k=0,1,\ldots$ by 
\begin{subequations} 
\label{GCG}
\begin{alignat}{3}
\gamma_k(t,x) &= f(t,x,\bar{m}_k) &\quad& \mbox{in } Q, \label{GCGg}\\ 
 -\del_t u_k - \nu \Delta{u}_k + H(t, x, \nabla{u}_k) &= \gamma_k(t,x) &\quad& \mbox{in } Q, \label{GCGa}\\ 
 u_k(T, x) &= g(x) && \mbox{on } \T^d,\label{GCGb} \\
\del_t m_k - \nu \Delta{m}_k - \nabla \cdot (m_k \nabla_p H(t, x, \nabla u_k)) &= 0 && \mbox{in } Q, \label{GCGc}\\ 
m_k(0, x) &= m_0(x) && \mbox{on } \T^d, \label{GCGd}
\end{alignat}
%\tag{MFG}
and $\bar{m}_{k+1}$ is updated by 
\begin{equation}
\bar{m}_{k+1}=(1-\delta_k) \bar{m}_k+\delta_k m_k\quad \mbox{in }Q,% && 
\label{GCGf}
\end{equation}
\end{subequations}
where $\delta_k\in (0,1]$ is a suitably chosen parameter. For example, if $\delta_k$ is chosen as $\delta_k = 1/(k+1)$, the updating rule is rewritten as 
\begin{equation*}
%\label{eq:fp1}
\bar{m}_{k+1} = \frac{1}{k+1} \sum_{j=0}^{k} m_j.
\end{equation*}
In this case, the GCG method \eqref{GCG} is referred to as the \emph{fictitious play} iterative method studied in \cite{CH17, YCLZ24} and, therefore, the GCG method is interpreted as a generalization of the fictitious-play method. \add{Fictitious play has been extensively studied in various areas of game theory. In particular, its application to finite-state mean field games has been investigated in \cite{HS19,PPLGEP20}.} \add{Anyway}, several alternative choices of $\delta_k$ may be employed to accelerate convergence.
Details will be presented in Subsection~\ref{sec:step}. In the GCG method, we first solve the terminal value problem \eqref{GCGg}-\eqref{GCGb} of the HJB equation in $Q$ using a given density distribution $\bar{m}_k$. Then, using the updated solution $u_k$ of the HJB equation, we solve the initial value problem \eqref{GCGc}-\eqref{GCGd} of the FP equation in $Q$ to update the density distributions $m_k$, and get $\bar{m}_{k+1}$ by \eqref{GCGf}.
This procedure is repeated until a solution of the MFG system is obtained. Thus, the computational cost remains large.
However, if the convergence is sufficiently fast, it can be more efficient to a few $d$-dimensional terminal and initial value problems rather than solving a $(d+1)$-dimensional problem at once.
Therefore, studying the convergence of the iteration in detail and proposing advantageous choices of $\delta_k$ is of practical significance.

In this context, Lavigne and Pfeiffer \cite{LP23} focused on the \emph{potential} MFG and succeeded in deriving explicit convergence rates of the GCG method in terms of the optimality gap $\eps_k$ (see Definition \ref{def:gap}) and the exploitability $\sigma_k$  (see Definition \ref{def:ex}).
They further proved that, depending on the choice of $\delta_k$, the optimality gap decays either exponentially or polynomially. Taken together, these results reveal almost all the essential aspects of the convergence behavior of the GCG method.
Here, the potential MFG refer to cases in which the interaction term admits a potential (see \textup{(f-P)} and Remark  \ref{rem:potential}).
In \cite{LP23}, the authors considered a very general (but smooth) Hamiltonian $H$, and for $f$ they assumed the following condition: 
\begin{equation} 
\label{coupling_term_LP23}
|f(t_2, x_2, m_2) - f(t_1, x_1, m_1)| \le C (|t_2-t_1|^{\alpha_0}+|x_2-x_1|) + L_f \|m_2-m_1\|_{L^2(\T^d)},
\end{equation}
where $C > 0$, $\alpha_0\in (0,1)$, $L_f > 0$ are suitable constants, $(t_1,x_1),(t_2,x_2)\in Q$, and $m_1,m_2\in\mathcal{D}_1(\mathbb{T}^d)$. 
However, this condition does not cover the locally coupled interactions \eqref{eq:ex-f} considered in the present work. On the other hand, Bonnans et al. \cite{B21} proved the existence and uniqueness of smooth solutions for the MFG system under a condition that  
\begin{equation} 
\label{eq:local0}
|f(t, x_2, m_2) - f(t, x_1, m_1)| \le \add{C_f} \left(|t_2-t_1|^{\alpha_0}+|x_2-x_1| + \|m_2-m_1\|_{L^\infty(\T^d)}^{\alpha_0}\right).
\end{equation}
Unfortunately, the local coupling terms we are interested in do not satisfy this condition either.
In view of these limitations, we consider coupling terms that satisfy the condition \textup{(f-L)} presented in Subsection \ref{sec:assumptions}. 
In fact, the interaction term $f$ given as \eqref{eq:ex-f} satisfies the condition \textup{(f-L)}. 
%, if $\varphi$ is a Lipschitz continuous function by itself. 
The objective of this work is to reconstruct the convergence analysis of \cite{LP23} by replacing the assumption on $f$ in \eqref{coupling_term_LP23} with \add{the assumption \textup{(f-L)}.}
However, this approach faces a major difficulty.
In the analysis, the uniform boundedness (i.e., boundedness independent of $k$) of the sequences generated by the GCG method plays a crucial role.
Under our assumptions, establishing such a uniform bound, in particular for 
$\|u_k\|_{W^{1,\infty}(Q)}$, is highly nontrivial.
To overcome this difficulty, we restrict the Hamiltonian to the most standard form:
\begin{equation}
\label{eq:quadH}
H(t, x, p) = \frac{1}{2}|p|^2 - h(t, x) \cdot p,\qquad (t,x,p)\in Q\times \mathbb{R}^d
\end{equation}
with a given function. ($|\cdot|$ and $\cdot$ denote the Euclid norm and inner-product in $\R^d$, respectively.) 
In this case, we can utilize the Cole--Hopf transformation $\phi=\exp(-u/2\nu)$ and rewrite the HJB equation as a reaction-diffusion equation for solving $\phi$. Then, thanks to parabolic regularity estimates, we are able to prove that $\|u_k\|_{W^{1,\infty}(Q)}$ is uniformly bounded. 
\add{
There is, however, a second reason for restricting ourselves to quadratic Hamiltonians. Namely, we need to assume that the corresponding Lagrangian $L(\cdot,\cdot,v)$ is strongly convex with respect to $v$; see \eqref{eq:L-s-convex} below. Indeed, the proof of Lemma \ref{LP23_lem19} relies on [32, Lemmas 19 and 20], and assumption \eqref{eq:L-s-convex} is essential in the derivation of these results.
} This is the reason that we impose such a specific structure on the Hamiltonian. \add{Nevertheless, a slight generalization is possible, namely, one may consider Hamiltonians of the form $H(t,x,p)=a(t,x)|p|^2-h(t,x)\cdot p$; see Remarks \ref{rem:th-a} and \ref{rem:rem-h} for the detail.} \add{Anyway}, within the settings, we are able to achieve our first objective.

We note that in \cite{LP23}, the authors refer to $f$ in \eqref{MFG_eq} as the first coupling term, and they additionally incorporate the effect of a price function, which they call the second coupling term. Since our interest in this paper lies in locally coupled interactions, we do not consider price interactions. See also Section \ref{sec:cr}, Item 4. 

\add{We now mention previous studies on the convergence of iterative methods for MFGs. Broadly speaking, there are two lines of work: one analyzes the convergence of iterative methods at the PDE level \cite{C21, CT22, CH17, LP23, TS24, YCLZ24}, while the other studies convergence after discretizing the MFG system \cite{AC10, A12, A13, ach13, I23, MR4596353, os25}. The work \cite{G12} can be regarded as belonging to both categories. Since the latter line of work also has to address discretization errors, its viewpoint is somewhat different from that of the present paper and is not directly comparable to our setting. Even in the former line of work, studies that establish explicit convergence rates, as in \cite{LP23}, appear to be limited to \cite{CT22, YCLZ24}. The policy iteration method in \cite{CT22} can handle local couplings, which are a main focus of this paper. However, as is also pointed out in \cite{LP23}, it is subject to restrictions such as the requirement that the coupling term be sufficiently small. This should be contrasted with the present work, where no smallness assumption on the coupling term is imposed.
}

\add{
The authors learned about Yu et al.~\cite{YCLZ24} after the completion of the present study. In \cite{YCLZ24}, explicit convergence rates are reported for the sequence generated by the GCG method for MFGs under consideration, without assuming the existence of a potential structure. (What they call fictitious play is essentially the GCG method.) As choices of $\delta_k$, they consider either $\delta_k={k_1}/({k+k_1})$ $(k_1>0)$, or a sufficiently small constant step size $\delta_k=\delta$. As is clear from \eqref{GCG}, the existence of a potential is not needed in order to generate the iterative sequence itself. However, both in \cite{LP23} and in the present paper, the cost functional $\mathcal{J}(m,w)$ arising from the variational formulation of the MFG is used in an essential way to derive estimates for the iterative sequence. The construction of this cost functional requires the existence of a potential of $f$.
By contrast, \cite{YCLZ24} presents explicit convergence rates without using $\mathcal{J}$. One of the key ingredients \cite[Proposition 3.11]{YCLZ24} is a certain geometric relation between a given density $m_\star$ and its best response $(\tilde{m},\tilde{w})$; see Remark \ref{rem:best} for this terminology. This proposition is likely to be a useful and valuable lemma not only in the numerical analysis of MFGs but also in their mathematical analysis. 

On the other hand, \cite{YCLZ24} imposes rather restrictive assumptions on the Hamiltonian $H$ and the coupling function $f$. First, it is assumed that the Benamou–Brenier-type integrand $m L(\cdot,\cdot,w/m)$ satisfies a condition called ``directional strong convexity'' \cite[Definition 3.10]{YCLZ24}. This is another, and indeed essential, key point that allows them to avoid using $\mathcal{J}$. The practical scope of this assumption is not entirely clear. The paper does not provide concrete classes of $H$ for which the assumption can be verified.
Moreover, although both global and local coupling functions are considered, the strong monotonicity \cite[Definition 3.8]{YCLZ24} of $f$ is used in an essential way. 
This assumption excludes many standard local couplings used in MFGs, such as \eqref{eq:ex-f} and $f(t,x,m)=m(t,x)^\alpha$ 
 with $\alpha>1$, although these couplings satisfy the classical Lasry–Lions monotonicity condition (f-M) below.
 
This is a major difference between the present paper and \cite{YCLZ24}. Although we assume the existence of a potential, our framework covers the most fundamental Hamiltonians and many local coupling functions that are meaningful in applications. We also note that \cite{YCLZ24} does not assume the boundedness of $f$, as we do. However, this is mainly due to the difference in the setting: in \cite{YCLZ24}, the density is considered under the constraint $0\leq m\leq 1$. Thus, this does not appear to be an essential difference. 

Finally, using the geometric properties of the best response, \cite{YCLZ24} proposes a backtracking line-search strategy as an adaptive choice of $\delta_k$; see \cite[Algorithm 2]{YCLZ24}. Its effectiveness is demonstrated numerically. However, no explicit convergence rate for the iterative sequence generated by this strategy seems to be established. Moreover, no numerical comparison appears to be made with the other adaptive step-size selection strategies considered in the present paper and in \cite{LP23}. These issues are certainly interesting, but they are beyond the scope of the present paper. See also Section \ref{sec:cr}, Item 5. 
}

Our second objective is to investigate the convergence behavior of the GCG method under local coupling term through various numerical experiments. Although numerical results were also reported in \cite{LP23}, they only considered cases without the effect of the first coupling term in their sense.
We discretize the GCG method by using the finite difference method proposed in \cite{I23}, and perform numerical computations.
We then evaluate, under locally coupled interactions, the relationship between the choice of $\delta_k$ and computational efficiency, as well as the validity of the theoretically established convergence rates. 
We do not, however, analyze the discretization error in this paper, as such a study would require additional technical developments. We therefore avoid addressing this issue here, and will provide a detailed investigation in a separate work \cite{NS26}.

Our third objective is to establish the existence and uniqueness of smooth solutions for the MFG system \eqref{MFG_eq} under the above problem setting.  This task also involves nontrivial difficulties. In particular, as in the convergence analysis of the GCG method, deriving a priori bounds for $\|\nabla u\|_{L^\infty(Q;\mathbb{R}^d)}$ becomes a major obstacle.
Although \cite{G16} has addressed this issue, the setting considered there is more restrictive than ours; namely, $H(t, x, p) = |p|^2/2$ and $f(t, x, m) = m(t,x)^{\alpha}$, where $\alpha$ must be taken sufficiently small depending on $d$. 
To the best of our knowledge, very limited progress has been made on the regularity theory of the MFG systems with locally coupled interactions (see, e.g., \cite{P14} for the existence of weak solutions). We overcome this difficulty by applying the Cole--Hopf transformation to convert the HJB equation into a reaction–diffusion equation, and then invoking parabolic regularity theory.

\add{
We briefly summarize the main contents and contributions of this paper.

The principal results of this paper are Theorems~\ref{thm_MFG_regularity},
\ref{LP23_prop23}, \ref{LP23_thm7},
\ref{LP23_thm8-1}, and \ref{LP23_thm8-2}.

\begin{itemize}
\item
Theorem~\ref{thm_MFG_regularity} establishes the existence and uniqueness of smooth solutions to the MFG system \eqref{MFG_eq}. This theorem is stated in Subsection~\ref{sec:mfg}, and its proof is given in Section~\ref{sec_MFG_regularity}.

\item
Theorem~\ref{LP23_prop23} establishes the well-posedness of the GCG method \eqref{GCG}. This theorem is stated in Subsection~\ref{sec:mfg}, and its proof is presented in Section~\ref{sec:proof-wp}.

\item
Theorem~\ref{LP23_thm7} provides a general convergence result for the GCG method. Theorems~\ref{LP23_thm8-1} and \ref{LP23_thm8-2} provide explicit convergence-rate estimates for the GCG method combined with predefined and adaptive step-size selection strategies, respectively. These theorems are stated in Subsection~\ref{sec:convergence}. Their proofs are given in Sections~\ref{sec:LP23_lem6} and \ref{sec:proof-rate}.
\end{itemize}

The remainder of this paper is organized as follows.

\begin{itemize}

\item
Section~\ref{sec:settings} introduces the problem setting considered in this paper. In particular, the function spaces are introduced in Subsection~\ref{sec:fs}, and the assumptions are stated in Subsection~\ref{sec:assumptions}.

\item
Section~\ref{sec:GCG} is devoted to the variational formulation of potential MFGs and the GCG method. In Subsection~\ref{GCG_for_MFG}, we recall the optimization-based formulation of potential MFGs and the corresponding GCG method. Step-size selection strategies are discussed in Subsection~\ref{sec:step}.

\item
Section~\ref{sec_numerical_experiment} presents numerical experiments illustrating the relationship between the choice of $\delta_k$ and computational efficiency, as well as the validity of the theoretically predicted convergence rates for locally coupled interactions.

\item
Finally, Section~\ref{sec:cr} contains concluding remarks and a summary of the main findings.

\end{itemize}
}

This paper draws significantly from \cite{LP23}.
Results that can be directly cited from that paper (or other references), or used with only minor modifications, are stated as \emph{Lemmas}. On the other hand, \emph{Theorems} and \emph{Propositions} refer to results whose proofs rely on new ideas developed in the present work.

\section{Problem settings}
\label{sec:settings}

\subsection{Function spaces}
\label{sec:fs}

Following \cite{LP23}, 
we use the function spaces and norms presented below. Let $X$ denote $\T^d$, $Q$ or $(0,T)$. We write $\del_i = \del_{x_i}$, $\del_i\del_j=\del_{x_i}\del_{x_j}$ and so on.
\begin{itemize}
\item For any $k \in \Z_{\ge 0}$, $\mathcal{C}^k(X; \R^d)$ denotes the space of functions of class $\mathcal{C}^k$ on $X$ valued in $\R^d$. We simply write $\mathcal{C} := \mathcal{C}^0$ and $\mathcal{C}^k(X) := \mathcal{C}^k(X;\R)$.
\item For any $\alpha \in (0, 1)$, the set of functions whose derivatives up to order $k$ are H\"{o}lder continuous of order $\alpha$ is denoted by $\mathcal{C}^{k + \alpha}(X;\R^d)$. 
The function $u$ of $\mathcal{C}(Q)$ belongs to $\mathcal{C}^{\frac{\alpha}{2}, \alpha}(Q)$ if and only if $u(\cdot, x) \in \mathcal{C}^{\frac{\alpha}{2}}(0, T)$ and $u(t, \cdot) \in \mathcal{C}^{\alpha}(\T^d)$. If $\del_t u, \del_i u, \del_i \del_j u \in \mathcal{C}^{\frac{\alpha}{2}, \alpha}(Q)$ ($1 \le i, j \le d$), we write $u\in \mathcal{C}^{1+\frac{\alpha}{2}, 2+\alpha}(Q)$. 

\item For any $p \in [1, \infty]$, $L^p(X)$ and $W^{k, p}(X)$ denote the standard Lebesgue and Sobolev spaces, respectively. 
We use the standard Bochner space $L^p(0,T;W^{k,p}(\mathbb{T}^d))$. 
%Note that $L^p(0,T;L^{p}(\mathbb{T}^d))$. 
\item Throughout this paper, we fix a constant $q\in\mathbb{R}$ satisfying 
\begin{equation*}
%\label{eq:q}
q > d+2.
\end{equation*}
Then, set 
$$
W^{1, 2, q}(Q) := W^{1, q}(Q) \cap L^q(0, T; W^{2, q}(\T^d)),
$$
and 
%which represents the space of functions that are weak differentiable once in time variable and twice in space. The norm is given by
$$
\|u\|_{W^{1, 2, q}(Q)} := \|u\|_{W^{1, q}(Q)} + \|u\|_{L^q(0, T; W^{2, q}(\T^d))} \qquad (u \in W^{1, 2, q}(Q)).
$$
\item The space of control inputs in the FP equation is defined as
$$
\Theta := \left\{v \in \mathcal{C}(Q; \R^d) \mid D_x v \in L^q(Q; \R^{d \times d})\right\}
$$
equipped with the norm
$$
\|v\|_{\Theta} := \|v\|_{L^{\infty}(Q: \R^d)} + \|D_x v\|_{L^q(Q; \R^{d \times d})} \qquad (v \in \Theta),
$$
where $D_x v = (\del_j v_i)_{1 \le i, j \le d} \in \R^{d \times d}$ is the Jacobi matrix of $v = (v_1, \dots, v_d)^{\top}$.
\item Define the space of coupling terms as
$$
\Gamma := L^{\infty}(0, T; W^{1, \infty}(\T^d))\cap \mathcal{C}(Q) \qquad %\left(\text{where $W^{0, 1, \infty}(Q) := L^{\infty}(0, T; W^{1, \infty}(\T^d))$}\right)
$$
equipped with the norm
$$
\|\gamma\|_{\Gamma} := \|\gamma\|_{L^{\infty}(0, T; W^{1, \infty}(\T^d))} \qquad (\gamma \in \Gamma).
$$
%Furthermore, for any $R > 0$, we define
%$$
%\Gamma_R := \left\{\gamma \in \Gamma \ ; \ \|\gamma\|_{\Gamma} \le R\right\}.
%$$
\item The set of density distributions on $\T^d$ is defined as
\begin{equation}
\label{eq:d1}
\mathcal{D}_1(\T^d) := \left\{m \in \mathcal{C}(\T^d) \mid  m \ge 0, \int_{\T^d} m(x)\, dx = 1\right\}.
\end{equation}
\item The following space will be used as a feasible set $\mathcal{R}$ of the variational problems introduced later: 
\begin{equation}
\label{eq:spaceR}
\mathcal{R} := \left\{(m, w) \in W^{1, 2, q}(Q) \times \Theta \mid  \begin{matrix}
\del_t m - \nu\Delta{m} + \nabla \cdot w = 0, \quad m(0) = m_0 \\
\exists v \in L^{\infty}(Q; \R^d) \text{ s.t. } w = mv
\end{matrix}\right\}.
\end{equation}
As will be stated in Lemma \ref{LP23_lem5} below, the space $\mathcal{R}$ is a convex set.  
\end{itemize}

\begin{rem}
In order to treat local coupling terms, a density distribuion $m$ must be defined at every point of $\T^d$. That is why we define $\mathcal{D}_1(\T^d)$ as the subset of $\mathcal{C}(\T^d)$.
\end{rem}

\begin{rem} 
\label{B21_lem12}
%See \cite[Lemma 12]{B21}.
We recall a well-known Sobolev's type inequality. There exist constants $\delta\in (0,1)$ and $C>0$ such that
$$
\|v\|_{\mathcal{C}^{\delta}(Q)} + \|\nabla{v}\|_{\mathcal{C}^{\delta}(Q; \R^d)} \le C \|v\|_{W^{1, 2, q}(Q)}
$$
for any $v \in W^{1, 2, q}(Q)$. See \cite[pages 80 and 342]{lsu68} or \cite[Lemma 12]{B21} for example.
\end{rem}

%By virtue of this inequality, we can deduce from a bound on $\|u\|_{W^{1, 2, q}(Q)}$ that $|u|$ and $|\nabla u|$ are uniformly bounded. In particular, we can obtain an estimate for the Lipschitz constant with respect to space. This fact is frequently used for the arguments below.

%Below we often omit the arguments $t, x$ of $H$ as long as there is no risk of confusion. Moreover, we simply write $u(t)=u(t, \cdot)$ for a function $u = u(t, x)$. 
Finally, unless otherwise stated, the letter $C$ denotes a generic positive constant independent of the iteration number $k$. Whenever we need to emphasize that the constant $C$ depends on the parameter $R$, we denote it by $C(R)$.

\subsection{Assumptions}
\label{sec:assumptions}

We make the following assumptions on $H$, $f$, $g$ and $m_0$. 

\begin{description} \label{assum-L}
\item[(H)] The Hamiltonian $H$ is given by \eqref{eq:quadH},  where $h \in \mathcal{C}^{1+\alpha_0}(Q; \R^d)$ with a constant $\alpha_0\in (0,1)$. 

\item[(f-B)] The coupling term $f$ is bounded. That is, there exists a constant $C_0$ satisfying:
\label{eq:f-123}
\begin{equation}
|f(t, x, m)| \le C_0\label{eq:f-1}
\end{equation}
for $(t, x)\in Q$, and $m\in \mathcal{D}_1(\T^d)$. 

\item[(f-L)] The coupling term $f$ is Lipschitz continuous. That is, there exist constants $C_0$, $L_f > 0$ and $\alpha_0 \in (0, 1)$ satisfying: 
\begin{subequations}
\begin{equation}
|f(t_2, x, m_2) - f(t_1, x, m_1)| \le C_0 |t_2-t_1|^{\alpha_0} + L_f \|m_2 - m_1\|_{L^{\infty}(\T^d)}
\label{eq:f-2}
\end{equation}
for \add{$(t_1, x),(t_2, x) \in Q$, and $m_1,m_2 \in \mathcal{D}_1(\T^d)$}.
Furthermore, if $m \in \mathcal{D}_1(\T)$ is a Lipschitz continuous function (of $x$) with the Lipschitz constant $\operatorname{Lip}_x(m)$, then $f(t,\cdot,m)$ is also a Lipschitz continuous function of $x$ with the Lipschitz constant $\operatorname{Lip}_x(m;f)$: 
\begin{equation}
\label{eq:f-3}
|f(t, x_2, m) - f(t, x_1, m)| \le \operatorname{Lip}_x(m;f) |x_2 - x_1|
\end{equation}
\end{subequations}
for $(t, x_1),(t, x_2) \in Q$. Moreover, if $\operatorname{Lip}_x(m)\le R$ for some $R>0$, there is a constant $C(R)$ depending on $R$ such that $\operatorname{Lip}_x(m;f)\le C(R)$. 

\item[(f-M)] The coupling term $f$ is monotone with respect to the third argument. That is, 
\begin{equation}
\label{eq:f-4}
\int_{\T^d} \left[f(t, x, m_2) - f(t, x, m_1)\right][m_2(x) - m_1(x)]\, dx \ge 0
\end{equation}
for $t \in [0, T]$ and $m_1,m_2 \in \mathcal{D}_1(\T^d)$. 
\item[(f-P)] The coupling term $f$ has a potential. That is, there exists a continuous function $F : [0, T] \times \mathcal{D}_1(\T^d) \to \R$ such that
\begin{equation}
\label{eq:f-5}
F(t, m_2) - F(t, m_1) = \int_0^1 \int_{\T^d} f(t, x, sm_2 + (1-s)m_1)[m_2(x)-m_1(x)]~ dx ds.
\end{equation}
\item[(TIV)] $g \in \mathcal{C}^{2 + \alpha_0}(\T^d)$ and $m_0 \in \mathcal{D}_1(\T^d) \cap \mathcal{C}^{2 + \alpha_0}(\T^d)$ with with a constant $\alpha_0\in (0,1)$. 
\end{description}

\begin{rem}
\label{rem:potential}
\add{When \textup{(f-P)} is in force, the MFG system \eqref{MFG_eq} is said to be potential.}
\end{rem}

\begin{rem}
\label{rem:constants}
The constants $\alpha_0$ appearing in \textup{(H)}, \textup{(f-L)} and \textup{(TIV)} are assumed to take the same value. 
The constants $C_0$ in \textup{(f-B)} and \textup{(f-L)} are assumed to take the same value too. 
\end{rem}

\begin{rem}
\label{rem:unique}
\add{The monotonicity condition \textup{(f-M)} is commonly used to ensure the uniqueness of solutions to the MFG system; see, for example, the proof of Theorem \ref{thm_MFG_regularity} in Section \ref{sec_MFG_regularity}. In addition, this condition is essential for the convergence analysis of the GCG method developed in this paper; see Proposition \ref{LP23_lem24}.}
\end{rem}

\begin{rem}
\label{prop:L_and_H}
Under the condition \textup{(H)}, the running cost $L$ in \eqref{eq:r-cost} turns out to be  
\begin{equation*}
%\label{eq:quadL}
L(t,x,v) = \frac{1}{2} |v-h(t,x)|^2, \qquad (t,x,v)\in Q\times \mathbb{R}^d. 
\end{equation*}
The function $L=L(t,x,v)$ is differentiable with respect to $v$, and its gradient $\nabla_v L$ is differentiable with respect to both $x$ and $v$. Moreover, $L, \nabla_v L, D_x (\nabla_v L)$, and $D_v (\nabla_v L)$ are H\"{o}lder continuous on any bounded subset of $Q$. Furthermore, 
there exists a constant $C>0$ satisfying 
\begin{subequations}
\begin{align}
L(t, x, v) & \le C(1 + |v|^2), \\
|L(t, x_2, v) - L(t, x_1, v)| &\le C |x_2-x_1| (1+|v|^2),\\
L(t, x, v_2) &\ge L(t, x, v_1) + \nabla_v L(t, x, v_1) \cdot (v_2 - v_1) + \frac{1}{2C} |v_2 - v_1|^2 \label{eq:L-s-convex}
\end{align}
\end{subequations}
for $(t, x) ,(t, x_1) ,(t, x_2) \in Q$, and $v,v_1,v_2 \in \R^d$. 
\end{rem}

\add{
\begin{example}
We present an example of a coupling function $f$ satisfying
\textup{(f-B)}, \textup{(f-L)}, \textup{(f-M)}, and \textup{(f-P)}.
Let $\tau : [0, T] \to \mathbb{R}$ be a H\"older continuous function with exponent
$\alpha_0 \in (0,1)$, and let
$\psi : \mathbb{T}^d \to \mathbb{R}$ be a Lipschitz continuous function.
Consider the coupling function
\[
f(t,x,m)
:=
\tau(t)+\psi(x)+C_1\min\{m(x)^\alpha,\beta\},
\quad
(t,x)\in Q,\ 
m\in\mathcal{D}_1(\mathbb{T}^d),
\]
where $\alpha\ge 1$, $\beta\ge 1$, and $C_1>0$ are given constants.
In the numerical experiments presented in
Section~\ref{sec_numerical_experiment}, we take $\tau\equiv 0$, $\psi(x)=|x-x_0|^2$, and $\alpha=1$, where
$x_0=\left(\frac12,\dots,\frac12\right)\in\mathbb{T}^d$.

We claim that $f$ satisfies all the conditions \textup{(f-B)}, \textup{(f-L)}, \textup{(f-M)}, and \textup{(f-P)}. 

First, it is straightforward to verify that \textup{(f-B)} holds.
Next, \textup{(f-L)} holds with $L_f=C_1\alpha\beta^{\frac{\alpha-1}{\alpha}}$, $\operatorname{Lip}_x(m;f)=
L_f\operatorname{Lip}_x(m)$, and therefore $C(R)=RL_f$.

The condition \textup{(f-M)} follows from the fact that the function
$\min\{\xi^\alpha,\beta\}$ of $\xi\in [0,\infty)$ is non-decreasing. 

Finally, we show that \textup{(f-P)} holds.
Define the function $\Psi:[0,\infty)\to\mathbb{R}$ by
\[
\Psi(\xi)
:=
\int_0^\xi \min\{y^\alpha,\beta\}\,dy.
\]
A direct computation yields
\[
\Psi(\xi)
=
\begin{cases}
\dfrac{\xi^{\alpha+1}}{\alpha+1},
&
0\le \xi\le \beta^{1/\alpha},
\\[2ex]
\beta \xi
-
\dfrac{\alpha}{\alpha+1}\beta^{1+1/\alpha},
&
\xi>\beta^{1/\alpha}.
\end{cases}
\]

We define
\[
F(t,m)
:=
\int_{\mathbb{T}^d}
\Big[
\psi(x)\,[m(x)-1]
+
C_1\Psi(m(x))
\Big]
\,dx.
\]

Since $\Psi$ is Lipschitz continuous on $[0,\infty)$,
the mapping $m\mapsto \Psi(m)$ is continuous in $\mathcal{D}_1(\mathbb{T}^d)$, and therefore $F$ is continuous on $\mathcal{D}_1(\mathbb{T}^d)$. 
Moreover, \eqref{eq:f-5} follows from
\[
\Psi(\xi_2)-\Psi(\xi_1)
=
\int_{\xi_1}^{\xi_2}
\min\{y^\alpha,\beta\}
\,dy
=
\int_0^1
\min\!\bigl\{
[s\xi_2+(1-s)\xi_1]^\alpha,
\beta
\bigr\}
(\xi_2-\xi_1)
\,ds,
\]
and
\[
\int_{\mathbb{T}^d}
\int_0^1
\tau(t)
[m_2(x)-m_1(x)]
\,ds\,dx
=
\int_0^1\tau(t)\,ds
\int_{\mathbb{T}^d}
[m_2(x)-m_1(x)]
\,dx
=
0,
\]
for all
$\xi_1,\xi_2\ge 0$,
$t\in[0,T]$,
and
$m_1,m_2\in\mathcal{D}_1(\mathbb{T}^d)$.

Although not required in the present paper,
it is worth noting that $F$ is convex with respect to $m$.
Indeed, $\Psi$ is convex and the remaining term is affine in $m$.
\end{example}
}

\subsection{The well-posedness of MFG}
\label{sec:mfg}

In \cite[Theorem 1]{B21} and its proof, the authors proved the following. 
If the Lipschitz continuity assumption \textup{(f-L)} is replaced by \eqref{eq:local0}, the MFG system \eqref{MFG_eq} admits a unique solution $(\bar{u},\bar{m})\in W^{1,2,q}(Q)\times W^{1,2,q}(Q)$ and, moreover, there exists a constant $\alpha\in (0,1)$ such that  
\begin{subequations}
    \label{eq:classical}
    \begin{gather}
(\bar{m}, \bar{v}, \bar{u}, \bar{\gamma}) \in \mathcal{C}^{1+\frac{\alpha}{2}, 2 + \alpha}(Q) \times \mathcal{C}^{\alpha}(Q; \R^d) \times \mathcal{C}^{1+\frac{\alpha}{2}, 2+\alpha}(Q) \times \mathcal{C}^{\alpha}(Q),\\
D_x \bar{v} \in \mathcal{C}^{\alpha}(Q; \R^{d \times d}),
\end{gather}
\end{subequations}
where we have set
\begin{equation}
\label{eq:vw}
\bar{v}=-\nabla_p H(t,x,\nabla \bar{u}),\quad 
\bar{\gamma}=f(\cdot,\cdot,\bar{m}).
%\bar{w}=\bar{m}\bar{v}=-m\nabla_p H(t,x,\nabla \bar{u}). 
\end{equation}

%\end{rem}

However, as was mentioned in the Introduction, the local coupling terms such as \eqref{eq:ex-f} do not satisfy \eqref{eq:local0} because $C_f$ must depend on $m$. 
Consequently, \cite[Theorem 1]{B21} itself and that proof are not directly applied to our setting. 
Instead, for a particular $H$ given in \textup{(H)}, we can establish the following theorem by making a delicate use of parabolic regularity estimates. The proof of this theorem is deferred to Section \ref{sec_MFG_regularity}. 

%That is, we have the following theorem. 

\begin{theorem} \label{thm_MFG_regularity}
Suppose that 
\textup{(H)}, 
\textup{(f-B)}, 
\textup{(f-L)}, 
\textup{(f-M)}, and \textup{(TIV)} are satisfied. Then, the MFG system \eqref{MFG_eq} admits a unique solution $(\bar{u},\bar{m})\in W^{1,2,q}(Q)\times W^{1,2,q}(Q)$. Moreover, $\bar{v}$ and $\bar{\gamma}$ defined as \eqref{eq:vw} are, respectively, functions of $\Theta$ and of $\Gamma$. Furthermore, they are classical solutions in the sense that \eqref{eq:classical} holds true with some $\alpha\in (0,1)$. 
\end{theorem}

\begin{rem}
The condition \textup{(f-P)} is not required for the proof of Theorem \ref{thm_MFG_regularity}. \add{The proof of this theorem makes use of the H\"{o}lder continuity of $D_x\nabla_p H=-D_x h$. For this reason, the assumption $h\in \mathcal{C}^{1+\alpha_0}(Q;\mathbb{R}^d)$ is required.}
\end{rem}

\begin{rem}
At present, we cannot prove the existence and uniqueness result as above for a general $L$ satisfying the regularity properties stated in Remark \ref{prop:L_and_H}. This is an important open problem. 
\end{rem}

\subsection{The well-posedness of the GCG method}
\label{sec:GCG}

As stated in the following theorem, the GCG method \eqref{GCG} is well-defined, and the sequences generated by \eqref{GCG} are all uniformly bounded. The proofs will be postponed in Section \ref{sec:proof-wp}.

\begin{theorem}%[Well-posedness of Algorithm \ref{alg_GCG}] 
\label{LP23_prop23}
Suppose that 
\textup{(H)}, 
\textup{(f-B)}, 
\textup{(f-L)}, 
\textup{(f-M)}, and \textup{(TIV)} are satisfied. 
Let $\bar{m}_0\in W^{1,2,q}(Q)$ be a function satisfying 
$\del_t \bar{m}_0 - \nu\Delta{\bar{m}_0}  = 0$ in $Q$ and $\bar{m}_0(0) = m_0$ on $\T^d$. 

\begin{itemize}
\item [\textup{(i)}] The GCG method \eqref{GCG} with the initial guess $\bar{m}_0$ generates  sequences 
\begin{equation*}
%\label{eq:seq1}
u_k\in W^{1,2,q}(Q),\quad m_k\in W^{1,2,q}(Q),\quad \gamma_k \in \Gamma,\quad 
\bar{m}_{k}\in W^{1,2,q}(Q)
\end{equation*}   
for $k=0,1,2,\ldots$. 

\item[\textup{(ii)}] These sequences are uniformly bounded. That is, there exists a constant $C > 0$ such that
%\begin{subequations}
%\label{eq:seq2}
\begin{gather*}
\|u_k\|_{W^{1, 2, q}(Q)}\le C,\quad \|m_k\|_{W^{1, 2, q}(Q)} \le C, \quad 
\|\gamma_k\|_{\Gamma} \le C,\quad \|\bar{m}_{k}\|_{W^{1, 2, q}(Q)}\le C,\\  \|w_k\|_{\Theta}\le C,\quad  \|v_k\|_{\Theta} \le C,
%\|\bar{w}_k\|_{\Theta} \le C,
\end{gather*}
%\end{subequations}
where 
\begin{equation}
\label{eq:vwk}
{v}_k=-\nabla_p H(t,x,\nabla {u}_k),\quad 
{w}_k={m}_k{v}_k=-m_k\nabla_p H(t,x,\nabla {u}_k). 
\end{equation}
\end{itemize}
\end{theorem}

\begin{rem}
The condition \textup{(f-P)} is not needed for the proof of this theorem, as well as that of Theorem \ref{thm_MFG_regularity}. 
\end{rem}

\begin{rem}
    \label{rem:m0}
In Theorem \ref{LP23_prop23}, the initial guess $\bar{m}_0$ may be chosen from a considerably wider class of functions.
More precisely, for any 
$v\in \mathcal{C}^{\alpha_0/2,1+\alpha_0}(Q;\mathbb{R}^d)$ with the same $\alpha_0$ in \textup{(TIV)}, a function $\bar{m}_0$ that satisfies the equation 
$\del_t \bar{m}_0 - \nu\Delta{\bar{m}_0}+\nabla\cdot (v\bar{m}_0)  = 0$ in $Q$ and $\bar{m}_0(0) = m_0$ on $\T^d$ can be used.
%since $m_0\in \mathcal{C}^{2+\alpha_0}(\T^d)$ and 
\end{rem}

\section{Main Results}
\label{sec:main}

If the coupling term $f$
has a potential, the MFG system \eqref{MFG_eq} admits a variational structure, which plays a crucial role in elucidating the convergence properties of the GCG method \eqref{GCG}.
From now on, we always assume that 
\textup{(H)}, 
\textup{(f-B)}, 
\textup{(f-L)}, 
\textup{(f-M)}, \textup{(f-P)}, and \textup{(TIV)}  
%Assumptions \ref{assum-L}, \ref{assum-f}, \ref{assum-P}, and \ref{assum-tv-iv} 
hold.
Then the existence and uniqueness theorem (Theorem \ref{thm_MFG_regularity}) can be applied, and thus we can consider the unique solution $(\bar{u},\bar{m})\in W^{1,2,q}(\Omega)\times W^{1,2,q}(\Omega)$ to \eqref{MFG_eq}. Moreover, $\bar{v}$ and $\bar{\gamma}$ are defined as \eqref{eq:vw}, and they, respectively, are functions of $\Theta$ and of $\Gamma$. 

\subsection{Variational formulations}
\label{GCG_for_MFG}
%First, we start by interpreting the MFG as an optimization problem:
We first recall that, under the potential assumption \textup{(f-P)}, the MFG system \eqref{MFG_eq} has a variational formulation for the cost functionals
\begin{align*}
\mathcal{J} &:= \mathcal{J}_1 + \mathcal{J}_2,\\
\mathcal{J}_1(m, w) &:= \iint_Q m(t, x) L\left(t, x, \frac{w(t, x)}{m(t, x)}\right) ~ dxdt + \int_{\T^d} g(x) m(T, x) \, dx,\\
\mathcal{J}_2(m) &:= \int_0^T F(t, m(t))~ dt.
\end{align*}
\add{Therein, we understood as
$$
m(t, x) L\left(t, x, \frac{w(t, x)}{m(t, x)}\right) = \begin{cases*} 0 & if $m(t, x) = 0$ and $w(t, x) = 0$, \\
+\infty & if $m(t, x) = 0$ and $w(t, x) \neq 0$.
\end{cases*}
$$
}

\begin{lemma}
\label{prop:variational-MFG}
Let $\bar{w}=-\bar{m}\nabla_p H(t,x,\nabla \bar{u})$. \add{Then, $(\bar{m}, \bar{w})\in\mathcal{R}$ is a unique minimizer of $\mathcal{J}$. That is, }
\begin{equation} \label{minimizeJ}
\mathcal{J}(\bar{m}, \bar{w})=
\min_{(m, w) \in \mathcal{R}}{\mathcal{J}(m, w)}. \tag{$\mathcal{P}$}
\end{equation}
\end{lemma}%[Potential problem for MFG]

This lemma is proved essentially in the same way as \cite[Proposition 2]{B21}, and we do not state it here. In fact, the Lipschitz continuity assumptions \eqref{eq:f-2} and \eqref{eq:f-3} are not required. 

Accordingly, the GCG method can also be interpreted in a variational framework.

\add{
Let $m_{\star} \in W^{1, 2, q}(Q)$ with 
$m_\star(t,\cdot)\in \mathcal{D}_1(\T^d)$ $(t\in [0,T])$ be given, and set $\gamma = f(\cdot, \cdot, m_{\star}) \in \Gamma$. We introduce
$$
\tilde{\mathcal{Z}}[\gamma](m, w) := \mathcal{J}_1(m, w) + 
\iint_Q \gamma(t, x) [m(t, x)-m_{\star}(t, x)] ~ dx dt .
$$
As a mater of fact, the second term of $\tilde{\mathcal{Z}}[\gamma](m, w)$ is nothing but 
the partial linearization of $\mathcal{J}_2(m)$ in the sense of the G\^{a}teaux derivative at $m_{\star}$ in the direction $m-m_{\star}$:
$$
D\mathcal{J}_2[m_{\star}](m) = \lim_{\theta \to +0} \frac{\mathcal{J}_2(m_{\star} + \theta (m-m_{\star})) - \mathcal{J}_2(m_{\star})}{\theta}.
$$
From the uniformly boundedness of $f$ and Lebesgue's dominated convergence theorem, we see that the above limit converges to $\iint_Q \gamma (m-m_{\star})\, dx dt$.

In what follows, however, we are interested only in the minimization problem for $\tilde{\mathcal{Z}}[\gamma]$ with a fixed $\gamma$, as well as in differences of the form
$$
\tilde{\mathcal{Z}}[\gamma](m_2, w_2)-\tilde{\mathcal{Z}}[\gamma](m_1, w_1).
$$
The term $\iint_Q \gamma m_{\star}\, dx dt$ is therefore irrelevant, since it does not affect the minimization problem and cancels out in such differences. Hence, we omit this term and instead work with the following reduced form:
$$
\mathcal{Z}[\gamma](m, w) := \mathcal{J}_1(m, w) + 
\iint_Q \gamma(t, x) m(t, x) ~ dx dt .
$$
}

%\underbrace{\iint_Q \gamma(t, x) m(t, x) \, dx dt }_{\text{partial linearization of $\mathcal{J}_2(m)$}}.

\begin{prop} \label{LP23_lem6}
\add{Let $m_{\star} \in W^{1, 2, q}(Q)$ with $m_\star(t,\cdot) \in \mathcal{D}_1(\T^d)$ $(t\in [0,T])$ be given. Set $\gamma = f(\cdot, \cdot, m_{\star}) \in \Gamma$.} Then, there exists a unique $(\tilde{u},\tilde{m})\in  W^{1,2,q}(\Omega)\times W^{1,2,q}(\Omega)$ satisfying
\begin{subequations} 
\label{best}
\begin{alignat}{3}
 -\del_t \tilde{u} - \nu \Delta{\tilde{u}} + H(t,x, \nabla{\tilde{u}}) &= \gamma &\quad& \mbox{in } Q, \label{eq:best-a}\\ 
 \tilde{u}(T,x) &= g(x) && \mbox{on } \T^d,\label{eq:best-b}\\
\del_t \tilde{m} - \nu \Delta{\tilde{m}} - \nabla \cdot (\tilde{m} \nabla_p H(t, x, \nabla \tilde{u})) &= 0 && \mbox{in } Q, \label{eq:best-c}\\ 
\tilde{m}(0,x) &= m_0(x) && \mbox{on } \T^d. \label{eq:best-d}
\end{alignat}
\end{subequations}
Moreover, $(\tilde{m}, \tilde{w})=(\tilde{m}, - \tilde{m}\nabla_p H(\cdot,\cdot,\nabla \tilde{u}))$ is the unique solution of 
\begin{equation} \label{minimizeZ}
\mathcal{Z}[\gamma](\tilde{m}, \tilde{w})=
\min_{(m, w) \in \mathcal{R}} {\mathcal{Z}[\gamma](m, w)}. \tag{$\mathcal{P}[\gamma]$}
\end{equation}
%(\ref{minimizeZ}).
%$$
%(\bsym{m}[\gamma], \bsym{w}[\gamma]) = \argmin{(m, w) \in \mathcal{R}}{\mathcal{Z}[\gamma](m, w)}.
%$$
%In particular, 
\end{prop}

The proof of this proposition will be given in Section \ref{sec:LP23_lem6}. \add{We emphasize that system \eqref{best} is decoupled, because the coupling term $\gamma$ is given in advance. Unlike the original MFG system \eqref{MFG_eq}, the equations in \eqref{best} can therefore be solved separately and successively.}

\begin{algorithm}[t]
 \caption{Variational formulation of GCG method for MFG}
 \begin{algorithmic}[1]
 \State Choose an initial guess $(\bar{m}_0, \bar{w}_0) \in \mathcal{R}$.
 \ForAll {$k = 0, 1, 2, \dots$}
 \State Set $\gamma_k = f(\cdot, \cdot, \bar{m}_k)$.
 \State Find the solution $(m_k, w_k)$ to the optimization problem ($\mathcal{P}[\gamma_k]$). That is, using the solution $(\tilde{m},\tilde{w})$ of \eqref{best} with $\gamma=\gamma_k$, we set 
 \[
 u_k=\tilde{u},\quad 
 m_k=\tilde{m},\quad 
v_k=- \nabla_p H(\cdot,\cdot,\nabla {u}_k),\quad  
 w_k=m_kv_k.
 \]
 %That is, set
 %$$
 %u_k = \bsym{u}[\gamma_k], \quad v_k = \bsym{v}[\gamma_k], \quad m_k = \bsym{m}[\gamma_k], \quad w_k = \bsym{w}[\gamma_k].
 %$$
 \State Choose $\delta_k \in [0, 1]$ and set  $$(\bar{m}_{k+1}, \bar{w}_{k+1}) = (1-\delta_k) (\bar{m}_k, \bar{w}_k) + \delta_k (m_k, w_k).$$
 \EndFor
\end{algorithmic}
\label{alg_GCG}
\end{algorithm}

Based on Proposition \ref{LP23_lem6}, we derive a variational formulation of the GCG method as in Algorithm \ref{alg_GCG}. The original GCG method generates a sequence $(u_k,m_k,\gamma_k,\bar{m}_{k+1})$,
 while its variational counterpart is described in terms of $(\bar{m}_k,\bar{w}_k,m_k,w_k,\gamma)$, with $u_k$ and $v_k$ acting as auxiliary variables.

\begin{rem}
Although $\bar{w}_k$ does not affect the iteration itself, it is included to highlight the underlying variational structure. Moreover, introducing $\bar{w}_k$ is essential for the adaptive step-size selection methods (S1), (S2), and (S3) below.
\end{rem}

\add{
\begin{rem}
\label{rem:best}
In \cite[Definition 3.2]{YCLZ24}, 
$(\tilde{m},\tilde{w})$ in Proposition \ref{LP23_lem6} is called the best response of $m_\star$. 
\end{rem}
}
\subsection{Step-size selection methods}
\label{sec:step}
The choice of the step-size $\delta_k$ plays a significant role in the convergence of the GCG method. To achieve higher convergence, we should choose $\delta_k$ appropriately. 

\add{To do so, the following quantity plays a crucial role. 

\begin{definition}[Exploitability]
Letting $(\bar{m}_k,\bar{w}_k)$, $(m_k,w_k)$ and $\gamma_k$ be those of Algorithm \ref{alg_GCG}, we set 
\begin{align}
\sigma_k 
&= \mathcal{Z}[\gamma_k](\bar{m}_k, \bar{w}_k) - \mathcal{Z}[\gamma_k](m_k, w_k)\nonumber\\
&= \mathcal{Z}[\gamma_k](\bar{m}_k, \bar{w}_k) - \min\limits_{(m, w) \in \mathcal{R}}
\mathcal{Z}[\gamma_k](m, w) \ge 0,\label{eq:exploitability}
\end{align}
which we call the \emph{exploitability} at the $k$th iteration. 
\label{def:ex}
\end{definition}}

We first recall the following three selection methods, which are studied in \cite{LP23}.

\begin{description}
\item[(S1)] \emph{QAG condition-based step-sizes}. Let $c \in (0, 1)$ if $d=1$, and $c \in (0, 1/2]$ if $d\ge 2$. Let $\tau \in (0, 1)$. We say that $\delta \in [0, 1]$ satisfies the \textit{Quasi--Armijo--Goldstein} (QAG) condition if it satisfies
\begin{equation} \label{QAGcond}
\mathcal{J}(\bar{m}_k^{\delta}, \bar{w}_k^{\delta}) \le \mathcal{J}(\bar{m}_k, \bar{w}_k) - c \delta \sigma_k, 
\end{equation}
where 
\begin{equation}
\label{eq:mdel}
(\bar{m}_k^{\delta}, \bar{w}_k^{\delta}) = (1-\delta) (\bar{m}_k, \bar{w}_k) + \delta (m_k, w_k).
\end{equation}
Then, choose $\delta_k$ as
$$
\delta_k = \tau^{i_k}, \qquad i_k = \mathrm{argmin}\,\{i \in \Z_{>0} \mid \mbox{$\delta = \tau^i$ satisfies \eqref{QAGcond}}\}.
$$
(It should be kept in mind that we have $\delta_k\le \tau$.) 
\item[(S2)] \emph{Optimal step-sizes}. With $(\bar{m}_k^{\delta}, \bar{w}_k^{\delta})$ defined above, choose $\delta_k$ as
\begin{equation}
\label{optimize_J(delta)}
\delta_k = \argmin{\delta \in [0, 1]}{\mathcal{J}(\bar{m}_k^{\delta}, \bar{w}_k^{\delta})}.    
\end{equation}
\item[(S3)] \emph{Exploitability-based step-sizes}. 
Choose $\delta_k$ as
$$
\delta_k = \min{\left\{1, \ \frac{\sigma_k}{2 L_f D_k}\right\}},
$$
where
\begin{equation}
\label{eq:s3-dk}
\add{D_k := \int_0^T \|m_k(t) - \bar{m}_k(t)\|_{L^1(\T^d)} \|m_k(t) - \bar{m}_k(t)\|_{L^\infty(\T^d)}\,dt}.
\end{equation}
%with $L_f$ appearing the Lipschitz constant of $f$ in Assumptionm \ref{assum-f}. 
\end{description}

These three selection methods are chosen appropriately at each iteration. In this sense, we collectively refer to them as \textit{adaptive step-sizes}. On the other hand, it is conceivable to preselect $\delta_k$ in advance. That is, we consider the following predefined step-sizes:

\begin{description}
\item[(S4)] \emph{Predefined step-sizes}. Let $k_1$ and $k_2$ be constants such that $1 \le k_2 \le k_1$ and choose 
$$
\delta_k = \frac{k_2}{k+k_1}.
$$
\end{description}

\subsection{Explicit rates of convergence}
\label{sec:convergence}

We are now in a position to state our main results on the rate of convergence of the GCG method. 
To do so, we need the following quantity. 

\begin{definition}[The optimality gap]
\label{def:gap}
Letting $(\bar{m}_k,\bar{w}_k)$ be that of Algorithm \ref{alg_GCG}, we set 
\begin{align}
\eps_k &:= \mathcal{J}(\bar{m}_k, \bar{w}_k) - \mathcal{J}(\bar{m}, \bar{w}) \nonumber \\
&=\mathcal{J}(\bar{m}_k, \bar{w}_k)-
\min\limits_{(m, w) \in \mathcal{R}}{\mathcal{J}(m, w)}
\ge 0,\label{eq:optimalitygap}
\end{align}
which we call the \emph{optimality gap} at the $k$th iteration. 
\end{definition}

\begin{theorem} \label{LP23_thm7}
Let $(\bar{m}_{k},\bar{w}_{k},m_k,w_k,\gamma_k)$, $u_k$ and $v_k$ be the sequences generated by the GCG method \eqref{GCG} (Algorithm \ref{alg_GCG}) for a given $(\bar{m}_{0},\bar{w}_{0})\in\mathcal{R}$. 
Let $\sigma_k$ and $\eps_k$ be the exploitability and optimality gap defined in 
\eqref{eq:exploitability} and \eqref{eq:optimalitygap}, respectively. 
We choose $r \in \mathbb{R}$ by
\begin{equation}
\label{eq:val-p}
r =
\begin{cases}
2 & (d = 1)\\
\text{any number greater than } d & (d \ge 2).
\end{cases}
\end{equation}
Then, there exists a constant $C > 0$ such that for all $k = 0, 1, 2, \dots$,
\begin{subequations}
\begin{align}
\|\bar{m}_{k}-\bar{m}\|_{{L^2(0, T; L^{\infty}(\T^d))}} + \|\bar{w}_{k} - \bar{w}\|_{L^2(Q; \R^d)} & \le C \eps_{k}^{{\frac{1}{r}}}, \label{LP23_thm_ineq1}\\
\|\gamma_k - \bar{\gamma}\|_{{L^2(0, T; L^{\infty}(\T^d))}} &\le C \eps_k^{{\frac{1}{r}}}, \label{LP23_thm_ineq2}\\
\|u_k-\bar{u}\|_{L^{\infty}(Q)} &\le C \eps_k^{{\frac{1}{r}}}, \label{LP23_thm_ineq3}\\
\|m_k-\bar{m}_k\|_{{L^2(0, T; L^{\infty}(\T^d))}} + \|w_k - \bar{w}_k\|_{L^2(Q; \R^d)} & \le C \sigma_k^{{\frac{1}{r}}}.  \label{LP23_thm_ineq4}
\end{align}
\end{subequations}
%where $(\bar{m}, \bar{v}, \bar{u}, \bar{\gamma})$ is the solution of (\ref{MFG_eq}).  
Finally, we have $\eps_k\le C$. 
\end{theorem}

\begin{theorem} \label{LP23_thm8-1}
Suppose that the step-size $\delta_k$ is chosen in accordance with either adaptive stepsizes, that is \textup{(S1)}, \textup{(S2)} or \textup{(S3)}. 
Then, under the same assumptions of Theorem \ref{LP23_thm7}, we have the following. If $d=1$, there exists a constant $\lambda \in (0, 1)$ such that for all $k = 0, 1, 2, \dots$,
\begin{equation}
\label{eq:rate-d1}
\eps_k \le \eps_0 \lambda^k.
\end{equation}
On the other hand, if $d \ge 2$, there exist constants $N, M>0$ such that for all $k = 0, 1, 2, \dots$,
\begin{equation}
\label{eq:rate-d2}
\eps_k \le \frac{M}{(k+N)^{\frac{1}{\rho}}} 
%\qquad \left(r = 1-\frac{2}{p(p-1)} < 1\right).
\end{equation}
where 
\begin{equation}
\label{eq:const-r}
\rho = 1-\frac{2}{r(r-1)} < 1.
\end{equation}
\end{theorem}

\begin{theorem} \label{LP23_thm8-2}
Suppose that the step-size $\delta_k$ is chosen in accordance with the prescribed step-size \textup{(S4)}. 
Then, 
under the same assumptions of Theorem \ref{LP23_thm7}, 
there exists a constant $C_{\ast} > 0$ satisfying
\begin{equation} 
\label{thm8-2_eq}
\eps_{k+1} \le (1-\delta_k) \eps_k + C_{\ast} \delta_k^2 \eps_k^{\frac{2}{r(r-1)}}.
\end{equation}
Moreover, we have
\begin{equation} 
\label{thm8-2_eq_2}
\eps_k \le \frac{N}{(k + k_1)^s},
\end{equation}
where $s$, $N$ are constants defined by $s = k_2, N = \eps_0 k_1^{k_2} \exp{\left(C_{\ast} k_2^2(k_1+1)/k_1^2\right)}$ if $d = 1$, and by 
$$
s < \min{\left\{k_2, \frac{1}{\rho}\right\}}, \qquad N = \max{\left\{\eps_0 k_1^s, \left(\frac{C_{\ast} k_2^2}{k_2-s}\right)^{\frac{1}{r}}\right\}}, 
%\qquad \left(r = 1-\frac{2}{p(p-1)}\right)
$$
if $d \ge 2$. Therein, the constant $\rho$ is defined by \eqref{eq:const-r}. 
\end{theorem}

\add{The proofs of these three theorems are postponed to Sections \ref{sec:LP23_lem6} and \ref{sec:proof-rate}. In particular, the proof of Theorem \ref{LP23_thm7} is given at the end of Section \ref{sec:LP23_lem6} (page \pageref{proofth3.6}), while the proofs of Theorems \ref{LP23_thm8-1} and \ref{LP23_thm8-2} are provided at the end of Section \ref{sec:proof-rate} (pages \pageref{proofthe3.7}--\pageref{proofthe3.8}).}

\begin{rem}
We now compare the main results of \cite{LP23} with those of the present paper.
\begin{enumerate}
    \item The estimates in \cite{LP23} do not suffer from dimensional constraints, whereas our results depend on the spatial dimension $d$; in particular, only slow convergence estimates are obtained when the $d$ is greater than or equal to two. 
This technical limitation stems from the fact that the embedding $W^{1, r}(\T^d) \hookrightarrow L^{\infty}(\T^d)$ depends on $d$, and thus improving the estimates appears difficult.
More specifically, the convergence rate given in Theorem \ref{LP23_thm7} corresponds to the rate in \cite[Theorem~7]{LP23}, in which we can take $r=2$ for any $d$.
Regarding adaptive step-size rules, \cite[Theorem~8]{LP23} shows exponential convergence in any $d$, while our result, Theorem \ref{LP23_thm8-1}, only provides a polynomial decay when $d\ge 2$.

\item The errors $\bar{m}_{k}-\bar{m}$, $\gamma_k - \bar{\gamma}$, and ${m}_{k}-\bar{m}_k$ are estimated using the norms of $L^\infty(0,T;L^2(\T^d))$, $L^\infty(Q)$, and $L^\infty(0,T;L^2(\T^d))$, respectively, in \cite[Theorem~7]{LP23}. 
On the other hand, we use the norm of $L^2(0,T;L^\infty(\T^d))$ for estimating these errors. 

\item Furthermore, when the predefined step-sizes is considered, they establish the inequality
$$
\eps_k \le \eps_0 \exp{\left[\sum_{j=0}^{k-1} (C\delta_j^2 - \delta_j)\right]}
$$
in \cite[Theorem~8]{LP23}.  
%in Theorem \ref{LP23_thm8-2}.
That is, in their setting, one can always take $s = k_2$ for any $d$.
Of course, since the problem settings are different, a direct comparison should be made with caution.
\end{enumerate}
\end{rem}

\begin{rem} \label{rem_thm8-2}
In Theorem \ref{LP23_thm8-2}, the convergence speed of $\eps_k$ seems to be faster as $k_2$ increases. However, since $k_1$ ($\ge k_2$) also increases, the coefficient $N$ grows rapidly as a result. Therefore, increasing $k_1$ and $k_2$ would not improve the convergence speed. We numerically verify this in Section \ref{sec_numerical_experiment}.
\end{rem}

\add{
\begin{rem}
    \label{rem:th-a}
      Theorems \ref{LP23_thm7}, \ref{LP23_thm8-1} and \ref{LP23_thm8-2} remain true if (H) is replaced by the following (H$'$). See Remark \ref{rem:rem-h} for the proof of this fact. 
      \begin{description}
      \item[(H$'$)] 
      The Hamiltonian $H$ is given by 
      \begin{equation}
      \label{eq:quadH2}
      H(t,x,p)=a(t,x)|p|^2-h(t,x)\cdot p,
      \end{equation}
      where $a \in \mathcal{C}^{1}(Q; \R)$ and $h \in \mathcal{C}^{1+\alpha_0}(Q; \R^d)$ with a constant $\alpha_0\in (0,1)$. Moreover, there exists a constant $a_{\min}>0$ such that $a(t,x)\ge a_{\min}$ for $(t,x)\in Q$. 
      \end{description}
On the other hand, the corresponding extension of Theorem \ref{thm_MFG_regularity} is not immediate. See Remark \ref{rem:extension_thm2.8}.
\end{rem}

}

\section{Proof of Theorem \ref{LP23_prop23}}%{prop_unibdd}}
\label{sec:proof-wp}

We first recall some preliminary results. 
Following \cite{LP23}, we use the mappings presented below. 

\begin{itemize}
\item For any $\gamma \in \Gamma$, we set $\bsym{u}[\gamma]:=u$, where $u$ denotes the solution of the terminal value problem for the HJB equation
$$
-\del_t u - \nu \Delta{u} + H(t, x, \nabla u) = \gamma\quad \mbox{in }Q , \qquad u(T, \cdot) = g\quad \mbox{on }\mathbb{T}^d.
$$
\item For any $\gamma \in \Gamma$, set $\bsym{v}[\gamma] := - \nabla_p H(t, x, \nabla \bsym{u}[\gamma])$.
\item For any $v \in \Theta$, set $\bsym{M}[v]:=m$, where $m$ denotes the solution of the initial value problem for the FP equation
$$
\del_t m - \nu \Delta{m} + \nabla \cdot (m v) = 0\quad \mbox{in }Q, \qquad m(0, \cdot) = m_0\quad \mbox{on }\mathbb{T}^d.
$$
\item For any $\gamma \in \Gamma$, set $\bsym{m}[\gamma] := \bsym{M}[\bsym{v}[\gamma]]$.
\item For any $\gamma \in \Gamma$, set $\bsym{w}[\gamma] := \bsym{m}[\gamma] \bsym{v}[\gamma]$.
\item For any $m \in W^{1, 2, q}(Q)$, set $\bsym{\gamma}[m] := f(\cdot, \cdot, m)$.
\end{itemize}

Those mappings are actually well-defined and their regularities are summarized as follows. 

\begin{lemma} \label{LP23_lem13}
Let $R > 0$ and let $v \in \Theta$ be such that $\|v\|_{\Theta} \le R$. Then, $\bsym{M}[v] \in W^{1, 2, q}(Q)$ is well-defined, \add{that is, $\bsym{M}[v]$ exists uniquely}. Moreover, we have $\bsym{M}[v] \ge 0$ and $\displaystyle{\int_{\T^d} \bsym{M}[v](t, x)\, dx = \int_{\T^d} m_0(x)\, dx = 1}$. Furthermore, there exists a constant $C(R) > 0$, independent of $v$, such that $\|\bsym{M}[v]\|_{W^{1, 2, q}(Q)} \le C(R)$. 
\end{lemma}

\begin{lemma}\label{LP23_lem15}
Let $R > 0$. Then, there exists a constant $C(R) > 0$ such that for any $\gamma_1,\gamma_2 \in \Gamma$ with $\|\gamma_1\|_\Gamma,\|\gamma_2\|_\Gamma\le R$,  
$$
\|\bsym{u}[\gamma_2]-\bsym{u}[\gamma_1]\|_{L^{\infty}(Q)} \le C(R) \|\gamma_2-\gamma_1\|_{{L^2(0, T; L^{\infty}(\T^d))}}.
$$
\end{lemma}

\begin{lemma} \label{LP23_lem16}
For any $\gamma \in \Gamma$, $\bsym{u}[\gamma] \in W^{1, 2, q}(Q)$ is well-defined. Moreover, for any $R > 0$, there exists a constant $C(R) > 0$ such that for any $\gamma \in \Gamma$ with $\|\gamma\|_\Gamma\le R$, $\|\bsym{u}[\gamma]\|_{W^{1, 2, q}(Q)} \le C(R)$.
\end{lemma}

\begin{lemma} \label{LP23_lem18}
For any $m \in W^{1, 2, q}(Q)$, $\bsym{\gamma}[m] \in \Gamma$ is well-defined. Moreover, there exists a constant $C > 0$ such that for any $m_1,m_2 \in W^{1, 2, q}(Q)$,
\begin{align*}
\|\bsym{\gamma}[m_2]-\bsym{\gamma}[m_1]\|_{{L^2(0, T; L^{\infty}(\T^d))}} &\le C \|m_2 - m_1\|_{{L^2(0, T; L^{\infty}(\T^d))}}.
\end{align*}
\end{lemma}

\begin{lemma} \label{LP23_lem17}
For any $\gamma \in \Gamma$, 
$(\bsym{v}[\gamma], \bsym{m}[\gamma], \bsym{w}[\gamma]) \in \Theta \times W^{1, 2, q}(Q) \times \Theta$ is well-defined. Moreover, for any $R > 0$, there exists a constant $C(R) > 0$ such that for any $\gamma \in \Gamma$ with $\|\gamma\|_\Gamma\le R$,  
$$
\|\bsym{v}[\gamma]\|_{\Theta} + \|\bsym{m}[\gamma]\|_{W^{1, 2, q}(Q)} + \|\bsym{w}[\gamma]\|_{\Theta} \le C(R).
$$
\end{lemma}

\begin{lemma} \label{LP23_lem5}
$\mathcal{R}$ is a convex set. Moreover, for any $v \in \Theta$ and $m = \bsym{M}[v]$, we have $(\bsym{M}[v], v\bsym{M}[v]) \in \mathcal{R}$. 
\end{lemma}

See \cite[Lemmas 3, 5, 13-18]{LP23} for the proofs of the above six lemmas. In contrast to \cite[Lemmas 15 and 18]{LP23}, we have replaced the norm $\|\cdot\|_{L^{\infty}(0, T; L^2(\T^d))}$ with $\|\cdot\|_{L^2(0, T; L^{\infty}(\T^d))}$ in Lemmas \ref{LP23_lem15} and \ref{LP23_lem18}. But we can prove them in exactly the same way. 

\begin{rem}
Although it is shown in \cite[Lemma 18]{LP23} that $\|\bsym{\gamma}[m]\|_{\Gamma}\le C$ for $m\in W^{2,1,q}(Q)$, we cannot prove it under our assumptions \textup{(f-B)} and \textup{(f-L)} on $f$ at present. 
\end{rem}

We further recall the following results from \cite[Lemma 19]{LP23} and \cite[Theorems 4 and 7]{B21}. 

\begin{lemma} \label{LP23_lem19}
Let $\widehat{\gamma} \in \Gamma$ and define $\widehat{v} := \bsym{v}[\widehat{\gamma}], \widehat{m} := \bsym{m}[\widehat{\gamma}]$, and $\widehat{w} := \bsym{w}[\widehat{\gamma}]$. Then, for any $(m, w) \in \mathcal{R}$,
$$
\frac{1}{2C_0} \iint_Q |v-\widehat{v}|^2 m \, dx dt \le \mathcal{Z}[\widehat{\gamma}](m, w) - \mathcal{Z}[\widehat{\gamma}](\widehat{m}, \widehat{w}),
$$
where $v \in L^{\infty}(Q)$ is defined as a function satisfying $w = mv$.
\end{lemma}

\begin{lemma} 
\label{B21_thm4}
Let $R>0$ and suppose that 
$u_T\in W^{2,2-\frac{2}{q}}(\mathbb{T}^d)$, 
$b\in L^q(Q;\mathbb{R}^d)$, $c\in L^q(Q)$ and $w\in L^q(Q)$ satisfy
$$
\|u_T\|_{W^{2-\frac{2}{q}, q}(\T^d)} \le R, \quad \|b\|_{L^q(Q; \R^d)} \le R, \quad \|c\|_{L^q(Q)} \le R, \quad \|w\|_{L^q(Q)} \le R.
$$
Then, the terminal value problem 
\begin{equation} \label{parabolic_eq1}
\del_t u + \nu \Delta{u} + b \cdot \nabla{u} + c u = w \quad \mbox{in } Q, \qquad u(T, \cdot) = u_T \quad  \mbox{on } \T^d
%\tag{PE1}
\end{equation}
has a unique solution $u \in W^{1, 2, q}(Q)$ and there exists $C(R)>0$ such that $\|u\|_{W^{1, 2, q}(Q)} \le C(R)$.
\end{lemma}

\begin{lemma} \label{B21_thm7}
Let $R>0$ and suppose that 
$u_T\in \mathcal{C}^{2+\beta_0}(\T^d)$, 
$b\in \mathcal{C}^{\frac{\beta_0}{2}, \beta_0}(Q; \R^d)$, $c\in \mathcal{C}^{\frac{\beta_0}{2}, \beta_0}(Q)$ and $w\in \mathcal{C}^{\frac{\beta_0}{2}, \beta_0}(Q)$ satisfies
$$
\|u_T\|_{\mathcal{C}^{2+\beta_0}(\T^d)} \le R,\quad \|b\|_{\mathcal{C}^{\frac{\beta_0}{2}, \beta_0}(Q; \R^d)}\le R,\quad \|c\|_{\mathcal{C}^{\frac{\beta_0}{2}, \beta_0}(Q)} \le R, \quad \|w\|_{\mathcal{C}^{\frac{\beta_0}{2}, \beta_0}(Q)} \le R
$$
for some $\beta_0\in (0,1)$. 
Then, there exist constants $\beta=\beta(\beta_0) \in (0, 1)$ and $C=C(\beta_0, R) > 0$, and the terminal value problem  
\eqref{parabolic_eq1} has a unique solution $u \in \mathcal{C}^{1 + \beta/2, 2 + \beta}(Q)$. Moreover, we have $\|u\|_{\mathcal{C}^{1 + \beta/2, 2 + \beta}(Q)} \le C$.
\end{lemma}

\medskip

Having completed the above preparations, we are ready to proceed to the proof of Theorem \ref{LP23_prop23}. 

\begin{proof}[Proof of Theorem  \ref{LP23_prop23}, \textup{(i)}]
It is a direct consequence of Lemmas \ref{LP23_lem13}, \ref{LP23_lem15}, \ref{LP23_lem16}, and \ref{LP23_lem17}. 
We note that the initial guess $(\bar{m}_0,\bar{w}_0)$ with setting $\bar{w}_0=v\bar{m}_0$ described in Remark \ref{rem:m0} is an element of the set $\mathcal{R}$. 
\add{In fact, we fix a vector field $v$ satisfying
$$
\|v\|_{\mathcal{C}^{\frac{\alpha_0}{2}, \alpha_0}(Q; \R^d)} + \|\nabla \cdot v\|_{\mathcal{C}^{\frac{\alpha_0}{2}, \alpha_0}(Q)} \le C
$$
for some $\alpha_0 \in (0, 1)$.} Then, according to Lemma \ref{B21_thm7} (we interpret $t$ as $-t$), there exists $\beta=\beta(\alpha_0)\in (0,1)$ and $\bar{m}_0 \in \mathcal{C}^{1 + \beta/2, 2 + \beta}(Q)$ satisfying 
$$
\del_t \bar{m}_0 - \nu \Delta{\bar{m}_0} +v\cdot \nabla\bar{m}_0+(\nabla\cdot v)\bar{m}_0 = 0\quad \mbox{in }Q,\qquad  \bar{m}_0(0, \cdot) = m_0\quad\mbox{on }\T^d.
$$
This implies $(\bar{m}_0,\bar{w}_0)\in\mathcal{R}$, since $\bar{m}_0$ and $\bar{w}_0$ obviously belong to $W^{1,2,q}(Q)$ and $\Theta$, respectively.  
Particularly, we can take $v=0$, \add{in which case one may choose $\alpha_0=1$}. That is, we can take $(\bar{m}_0,0)\in\mathcal{R}$ as the initial guess as stated in the theorem, where $\bar{m}_0$ solves the above initial value problem with $v=0$. 
\end{proof}

The following proposition plays a crucial role in proving part (ii). 

\begin{prop}
\label{prop:du_k}
Under the assumptions of Theorem \ref{LP23_prop23}, the sequence $u_k$  defined there satisfies
\begin{equation}
\label{eq:bdd-gradu}
\|u_k\|_{W^{1,\infty}(Q;\mathbb{R}^d)}\le C, 
\end{equation}
\add{for $k=0,1,\ldots$.}
\end{prop}

\begin{proof}%[Proof of Proposition \ref{prop_unibdd}]
It is divided into two steps. 

\smallskip

\noindent \emph{Step 1.} We show \add{that there exists a positive constant $C$ such that}
\begin{equation}
\label{eq:bound-uk}
\|u_k\|_{L^{\infty}(Q)}\le C    
\end{equation}
\add{for $k=0,1,\ldots$.}
Let $x \in \T^d$ and let $(X_s)_{s \in [t, T]}$ be a solution to the following stochastic differential equation:
$$
\add{dX_s = \alpha_s \, ds + \sqrt{2 \nu} \, dB_s, \qquad X_t = x,}
$$
where $B_s$ denotes the $d$-dimensional Brownian motion, and $\alpha_s$ is a $\R^d$-valued control input at time $s$ such that \add{$\mathbb{E}\left[\int_t^T |\alpha_s|^2\, ds\right] < +\infty$}. Then, the solution $u_k$ of \eqref{GCGg}, \eqref{GCGa} and \eqref{GCGb} 
%\begin{equation} 
%\label{HJB_eq}
% -\del_t u_k - \nu \Delta u_k + \frac{1}{2}|\nabla u_k|^2 - h \cdot \nabla u_k = \gamma_k\quad  \text{in } Q, \qquad u_k(T) = g \quad  \text{on } \T^d
%\end{equation}
is interpreted as the value of the cost functional
$$
\mathcal{J}[\gamma_k](t, x, \alpha) := \mathbb{E} \left[\int_t^T [L(s, X_s, \alpha_s) + \gamma_k(s, X_s)]\, ds + g(X_T)\right],
$$
where $\alpha$ is chosen to minimize $\mathcal{J}[\gamma_k]$:
$$
u_k(t, x) = \inf_{\alpha} \mathcal{J}[\gamma_k](t, x, \alpha).
$$
Now, since
$$
\add{L(v) = \frac{1}{2} |v-h|^2 \ge \frac{1}{4} |v|^2 - \frac{1}{2}\|h\|_{L^\infty(Q;\R^d)}^2,
%\ge \frac{1}{2} |v|^2 - C,
}
$$
$\mathcal{J}[\gamma_k]$ is bounded from below as 
$$
\add{
\mathcal{J}[\gamma_k](t, x, \alpha) \ge \mathbb{E} \left[\int_t^T \left[\frac{1}{4} |\alpha_s|^2 - \frac{1}{2}\|h\|_{L^\infty(Q;\R^d)}^2- C_0\right]\, ds - \|g\|_{L^{\infty}(\T^d)}\right] \ge C.
}
$$
Moreover, considering the case $\alpha = 0$, we find an upper bound of $\mathcal{J}[\gamma_k]$ as
\begin{align*}
\mathcal{J}[\gamma_k](t, x, \alpha) &\le \mathcal{J}[\gamma_k](t, x, 0) = \mathbb{E} \left[\int_t^T \gamma_k(s, X_s)\, ds + g(X_T)\right]\\
&\le T C_0 + \|g\|_{L^{\infty}(Q)}.
\end{align*}
%%Therefore, $\|u_k\|_{L^{\infty}(Q)}$ is uniformly bounded.
Summing up, we prove \eqref{eq:bound-uk}. 

\smallskip

\noindent \emph{Step 2.} We apply the Cole--Hopf transformation $\phi_k = \exp{\left(-u_k/2 \nu\right)}$. Then, \eqref{GCGg}, \eqref{GCGa} and \eqref{GCGb} are equivalently written as:
\begin{equation} 
\label{CHHJB_eq}
\del_t \phi_k + \nu \Delta \phi_k + h \cdot \nabla \phi_k = \frac{1}{2 \nu} \gamma_k \phi_k \quad  \text{in } Q, \qquad \phi_k(T,\cdot) = \exp{\left(-\frac{g}{2 \nu}\right)} \quad \text{on } \T^d.
\end{equation}
As a result of Step 1, 
$\|\phi_k\|_{L^{\infty}(Q)}$, 
$\left\|\exp{\left(-{g}/{2\nu}\right)}\right\|_{W^{2-\frac{2}{q}, q}(\T^d)}$,  
$\|(1/2\nu)\gamma_k \phi_k\|_{L^q(Q)}$ are all bounded by a constant depending on $C$ in \eqref{eq:bound-uk}. 
We now apply Lemma \ref{B21_thm4} and obtain $\|\phi_k\|_{W^{1, 2, q}(Q)} \le C$. 
\add{This, together with Remark \ref{B21_lem12}, implies $\|\nabla \phi_k\|_{L^{\infty}(Q)} \le C$}.

On the other hand, we see from \eqref{eq:bound-uk} that $\phi_k \ge C$. \add{Therefore}, we deduce
$$
|\nabla u_k| = \left|-\frac{2 \nu}{\phi_k} \nabla \phi_k\right| \le C,
$$
which completes the proof.
\end{proof}

Now we can complete the proof of Theorem \ref{LP23_prop23}. 

\begin{proof}[Proof of Theorem \ref{LP23_prop23}, \textup{(ii)}]
We have already proved that \eqref{eq:bdd-gradu} holds. 
\add{We note that the constant $C$ appearing in this proof is taken to be independent of $k$. }
\begin{itemize}
    \item 
We apply Lemma \ref{B21_thm4} with $b =0$, $w = \gamma_k - H(\nabla u_k)$, and $u_T = g$. \add{As a result, we obtain $\|u_k\|_{W^{1, 2, q}(Q)} \le C$ and therefore $\|v_k\|_{\Theta} =\|\nabla_p H(\nabla u_k)\|_{\Theta} \le C$.}

\item \add{We infer from Lemma \ref{LP23_lem13} that $\|m_k\|_{W^{1, 2, q}(Q)} \le C$.}

\item If $\|\bar{m}_{k}\|_{W^{1, 2, q}(Q)}\le C$, we calculate as
$$
\|\bar{m}_{k+1}\|_{W^{1, 2, q}(Q)} \le (1-\delta_k) \|\bar{m}_k\|_{W^{1, 2, q}(Q)} + \delta_k \|m_k\|_{W^{1, 2, q}(Q)}\le C.
$$
Hence, we have $\|\bar{m}_k\|_{W^{1, 2, q}(Q)} \le C$ for all $k$ by induction. 

\item Because \add{$\sup_{t\in [0,T]}\operatorname{Lip}_x(\bar{m}_k(t, \cdot)) \le C \|\bar{m}_k\|_{W^{1, 2, q}(Q)} \le C$}, \eqref{eq:f-3} is available so that we have  $\|\gamma_k\|_{\Gamma}\le C$.

\item Similarly, we deduce $\|w_k\|_{\Theta} \le C$ from Lemma \ref{LP23_lem17}, and hence $\|\bar{w}_k\|_{\Theta} \le C$ by induction.

\item  Finally, since $(m_k, w_k) = (\bsym{M}[\bsym{v}[\gamma_k]], \bsym{M}[\bsym{v}[\gamma_k]] \bsym{v}[\gamma_k])$ and $\bsym{v}[\gamma_k] \in \Theta$ by Lemma \ref{LP23_lem17}, we have $(m_k, w_k) \in \mathcal{R}$ from Lemma \ref{LP23_lem5}. From the convexity of $\mathcal{R}$ and $(\bar{m}_0, \bar{w}_0) \in \mathcal{R}$, we also obtain $(\bar{m}_k, \bar{w}_k) \in \mathcal{R}$ by induction.
\end{itemize}
\end{proof}

\add{
\begin{rem}
\label{rem:rem-h}
One of the key points in the proof of Theorem \ref{LP23_prop23} is that $\|u_k\|_{W^{1,\infty}(Q)}$ can be bounded by a positive constant $C$ that depends on the upper bound $C_0$ of $\gamma_k=f(\cdot,\cdot,\bar m_k)$, but is independent of $\bar m_k$ itself.
If the $W^{1,\infty}(Q)$-estimate for the solution of the HJB equation depended on $\bar m_k$ through the forcing term $\gamma_k$, the above argument would no longer work. In the previous work \cite{LP23}, this difficulty is avoided by assuming global coupling \eqref{coupling_term_LP23}, which yields a uniform $W^{1,\infty}(Q)$-estimate for $u_k$ independent of $\bar m_k$; see \cite[Lemma 18]{LP23}. 
Several approaches are available for obtaining Lipschitz estimates for solutions of HJB equations, but the dependence on the forcing term is a crucial issue. The adjoint method introduced in \cite{E10} for first-order Hamilton--Jacobi (HJ) equations is particularly useful for this purpose. This method was later extended to viscous HJ equations in \cite{CG20}. In particular, \cite{CG20} considers the Hamiltonian given by, for example, 
\begin{equation}
\label{eq:quadH3}
H(t,x,p)=a(t,x)|p|^s-h(t,x)\cdot p,\quad s>1    
\end{equation}
and proves that the solution of the HJB equation
\[
-\partial_tu-\nu\Delta u+H(t,x,\nabla u)=\varphi(t,x)
\quad \text{in }Q,
\qquad
u(T,\cdot)=u_T
\]
satisfies
\[
\|u\|_{W^{1,\infty}(Q)}\le C.
\]
Here, we assume $a_{\min}\le a\in \mathcal{C}^1(Q)$ for some $a_{\min}>0$. 
In particular, the constant $C$ depends only on $\nu$, $a_{\min}$, $\|\varphi\|_{L^q(Q)}$, $\|u_T\|_{W^{1,\infty}(\mathbb{T}^d)}$, and the parameters $T$, $d$, and $q$. 
Hence, if Hamiltonian \eqref{eq:quadH} is replaced by \eqref{eq:quadH3}, the estimate \eqref{eq:bdd-gradu} holds. However, in order for a strong convexity \eqref{eq:L-s-convex} of $L$ to hold, we have to take $s=2$ in \eqref{eq:quadH3}.  
Consequently, as is noted in Remark \ref{rem:th-a}, Theorems \ref{LP23_thm7}, \ref{LP23_thm8-1} and \ref{LP23_thm8-2} remain valid if (H) is replaced by (H$'$). 
\end{rem}
}

\section{Proofs of Proposition \ref{LP23_lem6} and Theorem \ref{LP23_thm7}}
\label{sec:LP23_lem6}

Throughout this and subsequent sections, $r$ refers to the same constant appearing in Theorem \ref{LP23_thm7}.
We begin by showing the following propositions.

\begin{prop} \label{LP23_lem14}
Let $R > 0, v_i \in \Theta$ and set $m_i = \bsym{M}[v_i]$ $(i = 1, 2)$. Assume that $\|v_1\|_{\Theta} \le R$ and $\|m_2\|_{{W^{1, 2, q}(Q)}} \le R$. Then, there exists a constant $C(R) > 0$, depending only on $R$, such that
$$
\|m_2 - m_1\|_{{L^2(0, T; L^{\infty}(\T^d))}} \le C(R) \left(\iint_Q |v_2-v_1|^2 m_2 \,dx dt\right)^{{\frac{1}{r}}}.
$$ 
\end{prop}

\begin{proof}
Let $\mu := m_2-m_1 \in W^{1, 2, q}(Q)$ and $w := v_2-v_1 \in \Theta$. Then, $\mu$ is the solution to the equation
$$
\del_t \mu - \nu \Delta \mu + \nabla \cdot (v_1 \mu) = \vhi := -\nabla \cdot(w m_2)\quad \mbox{in }Q, \qquad \mu(0) = 0\quad \mbox{on }\mathbb{T}^d.
$$
Setting $V := H^1(\T^d)$ and $H := L^2(\T^d)$, we consider the Gel'fand triple $V \hookrightarrow H \hookrightarrow V'$, where $V'$ denotes the dual space of $V$. 
%Since $w \in \Theta$ is continuous, $w(t) \in L^2(\T^d; \R^d)$. In addition, since $\nabla \cdot w \in L^q(Q)$, $\nabla \cdot w(t) \in L^q(\T^d) \subset H$ for almost every $t$. Moreover, since $m_2 \in W^{1, 2, q}(Q)$ is sufficiently smooth, we have $\vhi(t) \in H$, and $\vhi(t)$ defines an element of $V'$. In particular, for $v \in V$,
%\begin{align*}{}_{V'}\langle \vhi(t), v \rangle_V &= (\vhi(t), v)_H = - \int_{\T^d} \left[\nabla \cdot (w(t)m_2(t))\right]v\,dx\\
%&= \int_{\T^d} [w(t) m_2(t)]\cdot \nabla v\, dx\\
%&\le \|w(t) m_2(t)\|_{L^2(\T^d; \R^d)} \|\nabla v\|_{L^2(\T^d; \R^d)}\\
%&\le \|w(t) m_2(t)\|_{L^2(\T^d; \R^d)} \|v\|_V.
%\end{align*}
We introduce the operator $A(t) : V\to V'$ by setting 
$$
A(t) m := - \nu \Delta m + \nabla \cdot \left(v_1(t) m\right) 
$$
for $m \in V$ and $t \in [0, T]$. 
Then, using $\|v_1\|_{\Theta} \le R$, we calculate as 
\begin{align*}
{}_{V'}\langle A(t)m, m\rangle_V 
%&= \int_{\T^d} \left\{\nu |\nabla m|^2 - \left[v_1(t) m\right] \cdot \nabla m\right\}\, dx\\
%&\ge \nu \|\nabla m\|_{L^2(\T^d; \R^d)}^2 - \|v_1\|_{L^{\infty}(Q; \R^d)} \|m\|_{L^2(\T^d)} \|\nabla m\|_{L^2(\T^d; \R^d)} \tag{Cauchy--Schwarz inequality}\\
&\ge \nu \|\nabla m\|_{L^2(\T^d; \R^d)}^2 - R \|m\|_H \|\nabla m\|_{L^2(\T^d; \R^d)}\\
&\ge \frac{\nu}{2} \|m\|_{H^1(\T^d)}^2 - \left(\dfrac{R^2}{2\nu}+\dfrac{\nu}{2}\right) \|m\|_H^2
\end{align*}
for $m\in V$, and 
%\begin{align*}
%{}_{V'}\langle B(t)m, m\rangle_V &\ge \frac{\nu}{2} \|\nabla m\|_{L^2(\T^d; \R^d)}^2 - \frac{R^2}{2\nu} \|m\|_H^2\\
%&= \frac{\nu}{2} \|m\|_{H^1(\T^d)}^2 - \lambda \|m\|_H^2. \tag{$\lambda = \dfrac{R^2}{2\nu}+\dfrac{\nu}{2}$}
%\end{align*}
%Additionally, for $m_1, m_2 \in V$,
\begin{equation*}
{}_{V'}\langle A(t)m_1, m_2\rangle_V
%=& \ \int_{\T^d} \left[\nu \nabla m_1 \cdot \nabla m_2 - \left(v_1(t) m_1\right) \cdot \nabla m_2\right]\, dx\\
%\le& \ \nu \|\nabla m_1\|_{L^2(\T^d; \R^d)} \|\nabla m_2\|_{L^2(\T^d; \R^d)} + %\|v_1(t)\|_{L^{\infty}(Q; \R^d)} \|m_1\|_H \|\nabla m_2\|_{L^2(\T^d; \R^d)}\\
\le (\nu + R) \|m_1\|_V \|m_2\|_V
\end{equation*}
for $m_1, m_2 \in V$. 

Therefore, since $\|\vhi(t)\|_{V'} \le \|w(t) m_2(t)\|_{L^2(\T^d; \R^d)}$, 
we can apply a well-known estimate for solutions to abstract parabolic equations 
(\cite[Chapter 3, Theorem 1.2]{L71} for example) to deduce
\begin{equation} \label{LP23_lem14_ineq1}
\|\mu\|_{L^2(0, T; V)} + \|\del_t \mu\|_{L^2(0, T; V')} \le C \|\vhi\|_{L^2(0, T; V')} \le C \|w m_2\|_{L^2(Q; \R^d)}.
\end{equation}
%Since
%$$
%\|w m_2\|_{L^2(Q; \R^d)}^2 = \iint_Q |(v_2-v_1) m_2|^2\,dx dt \le R \iint_Q |v_2-v_1|^2 m_2 %\,dx \,dt,
%$$
%applying the inequality (\ref{LP23_lem14_ineq1}), we have
This, together with $\|m_2\|_{{W^{1, 2, q}(Q)}}\le R$, implies
\begin{equation} \label{LP23_lem14_ineq2}
\|m_2-m_1\|_{L^2(0, T; H^1(\T^d))} \le C(R) \left(\iint_Q |v_2-v_1|^2 m_2 \,dx dt\right)^{\frac{1}{2}}.
\end{equation}

Now, we assume that $d\ge 2$ (and thus $r>d$). Set $W := W^{1, r}(\T^d)$. Since $\|v_1\|_{\Theta} \le R$, we see from Remark \ref{B21_lem12} and Lemma \ref{LP23_lem13} that there is a constant $C(R) > 0$ satisfying 
$$
\|\mu\|_{L^{\infty}(Q)} + \|\nabla \mu\|_{L^{\infty}(Q; \R^d)} \le 
C\|\mu\|_{W^{1, 2, q}(Q)} 
%\le \|m_2\|_{W^{1, 2, q}(Q)} + \|m_1\|_{W^{1, 2, q}(Q)} 
\le C(R).
$$
Hence, we have, for $t\in [0,T]$, 
\begin{align*}
\|\mu(t)\|_{L^r(\T^d)}^r %&= \int_{\T^d} |\mu(t, x)|^p \, dx\\
&= \int_{\T^d} |\mu(t)|^{2} \cdot |\mu(t, x)|^{r-2}\, dx\\
&\le \|\mu(t)\|_{L^{\infty}(\T^d)}^{r-2} \int_{\T^d} |\mu(t, x)|^2\, dx\le C(R)^{r-2} \|\mu(t)\|_{L^2(\T^d)}^2.
\end{align*}
Similarly, for $t\in [0,T]$, 
$$
\|\nabla \mu(t)\|_{L^r(\T^d; \R^d)}^r \le C(R)^{r-2} \|\nabla \mu(t)\|_{L^2(\T^d; \R^d)}^2.
$$
Summing up, 
$$
\|\mu(t)\|_W^r \le C(R) \|\mu(t)\|_{H^1(\T^d)}^2.
$$
As a consequence, we obtain from \eqref{LP23_lem14_ineq2} 
$$
\|m_2-m_1\|_{L^r(0, T; W)} \le C(R) \left(\iint_Q |v_2-v_1|^2 m_2 \,dx dt\right)^{\frac{1}{r}}.
$$
Since $W = W^{1, r}(\T^d) \hookrightarrow L^{\infty}(\T^d)$, this inequality implies
%$$
%\|m_2-m_1\||_{L^p(0, T;L^{\infty}(\T^d))} \le C \|m_2-m_1\|_{L^p(0, T; W)}.
%$$
%Therefore, we obtain
\begin{equation}
\label{eq:prop6.1}
\|m_2-m_1\|_{L^r(0, T; L^{\infty}(\T^d))} \le C(R) \left(\iint_Q |v_2-v_1|^2 m_2 \,dx dt\right)^{\frac{1}{r}}.
\end{equation}
For the case $d=1$ (and thus $r=2$), we directly obtain \eqref{eq:prop6.1}, since $H^{1}(\T^1) \hookrightarrow L^{\infty}(\T^1)$. 

The desired estimate is a direct consequence of \eqref{eq:prop6.1}. 
%sApplying $\|m_2-m_1\|_{L^2(0, T; L^{\infty}(\T^d))} \le C \|m_2-m_1\|_{L^p(0, T; L^{\infty}(\T^d))}$ to this inequality, we complete the proof.
\end{proof}

%Throughout the following discussion, $p$ refers to the same constant as in this lemma.

\begin{prop} \label{LP23_prop21}
Let $R > 0$ and $\widehat{\gamma} \in \Gamma_R$. Define $\widehat{m} := \bsym{m}[\widehat{\gamma}]$ and $\widehat{w} := \bsym{w}[\widehat{\gamma}]$. Then, there exists a constant $C(R) > 0$, independent of $\widehat{\gamma}$ such that for any $(m ,w) \in \mathcal{R}$ such that $\|m\|_{{W^{1, 2, q}(Q)}} \le R$, 
$$
\|m - \widehat{m}\|_{{L^2(0, T; L^{\infty}(\T^d))}} \le C(R) \sigma^{{\frac{1}{r}}}, \qquad
\|w - \widehat{w}\|_{L^2(Q; \R^d)} \le C(R) \left(\sqrt{\sigma} + \sigma^{{\frac{1}{r}}}\right),
$$
where $\sigma = \mathcal{Z}[\widehat{\gamma}](m, w) - \mathcal{Z}[\widehat{\gamma}](\widehat{m}, \widehat{w}) \ge 0$.
\end{prop}

\begin{proof}
We know $\iint_Q |v-\widehat{v}|^2 m \, dx dt \le 2 C_0 \sigma$ by Lemma \ref{LP23_lem19}, and $\|\widehat{v}\|_{\Theta} \le C(R)$ by Lemma \ref{LP23_lem17}. Using $\|m\|_{W^{1, 2, q}(Q)} \le R$ and applying Proposition \ref{LP23_lem14} with $R=C$, we derive the first inequality
$$
\|m-\widehat{m}\|_{L^2(0, T; L^{\infty}(\T^d))} \le C(R) \left(\iint_Q |v-\widehat{v}|^2 m\, dx dt \right)^{\frac{1}{r}} \le C \sigma^{\frac{1}{r}}.
$$
Moreover, 
%by decomposing $w-\widehat{w}$ as $m(v-\widehat{v}) + (m-\widehat{m})\widehat{v}$, we observe that
\begin{align*}
\|w-\widehat{w}\|_{L^2(Q; \R^d)} &\le \|m(v-\widehat{v})\|_{L^2(Q; \R^d)} + \|(m-\widehat{m})\widehat{v}\|_{L^2(Q; \R^d)}\\
&\le \sqrt{R} \left(\iint_Q |v-\widehat{v}|^2m\,dx dt\right)^{\frac{1}{2}} + C(R) \left(\iint_Q |m-\widehat{m}|^2\,dx dt\right)^{\frac{1}{2}}\\
&\le C \left(\sqrt{\sigma} + \|m-\widehat{m}\|_{L^2(0, T; L^{\infty}(\T^d))}\right)\\
&\le C \left(\sqrt{\sigma} + \sigma^{\frac{1}{r}}\right),
\end{align*}
which is the second inequality to be proved. 
\end{proof}

\begin{prop} \label{LP23_lem24}
Let $(m_i, w_i) \in \mathcal{R}$ ($i = 1, 2$) and define $\gamma_1 := \bsym{\gamma}[m_1]$. Then, we have
\begin{equation} 
\label{LP23_lem24_ineq1}
\iint_Q \gamma_1 (m_2 - m_1)\, dx dt \le \mathcal{J}_2(m_2) - \mathcal{J}_2(m_1)
\end{equation}
and
\begin{equation} 
\label{LP23_lem24_ineq1a}
\mathcal{J}_2(m_2) - \mathcal{J}_2(m_1) \le \iint_Q \gamma_1 (m_2 - m_1)\, dx dt
+ L_f \add{\int_0^T \|m_2(t) - m_1(t)\|_{L^1(\T^d)} \|m_2(t) - m_1(t)\|_{L^{\infty}(\T^d)}}\,dt.
\end{equation}
\end{prop}

\begin{proof}
Set
\begin{align*}
I &:= \left[\mathcal{J}_2(m_2) - \mathcal{J}_2(m_1)\right] - \left(\iint_Q \gamma_1(t, x)[m_2(t, x) - m_1(t, x)]\, dx dt\right)\\
&= \int_0^T \left(F(t, m_2(t)) - F(t, m_1(t)) - \int_{\T^d} f(t, x, m_1(t))[m_2(t, x) - m_1(t, x)]\,dx\right)\, dt.
\end{align*}
Since $F$ is the potential of $f$, we can express it as 
\begin{equation*}
I = \int_0^T \int_0^1 \int_{\T^d} \left[f(t, x, m_1(t) + s[m_2(t)-m_1(t)]) - f(t, x, m_1(t))\right] [m_2(t, x)-m_1(t, x)]\, dx ds dt.
\end{equation*}
Rewriting $m_2-m_1$ as $\frac{1}{s} \left\{[m_1+s(m_2-m_1)] - m_1\right\}$ and then applying the monotonicity \eqref{eq:f-4} of $f$, we see $I \ge 0$, which is \eqref{LP23_lem24_ineq1}. On the other hand, 
\begin{align*}
I &\le L_f \int_0^T \int_0^1 \int_{\T^d} s \|m_2(t)-m_1(t)\|_{L^{\infty}(\T^d)} |\add{m_2(t, x) - m_1(t, x)}| \, dx ds dt\\
&\le L_f \int_0^T \add{\|m_2(t) - m_1(t)\|_{L^1(\T^d)} \|m_2(t) - m_1(t)\|_{L^{\infty}(\T^d)}}\, dt,
\end{align*}
which is \eqref{LP23_lem24_ineq1a}. 
\end{proof}

\begin{prop} \label{LP23_cor25}
Under the same conditions as in Proposition \ref{LP23_lem24}, we have
$$
\mathcal{Z}[\gamma_1](m_2, w_2) - \mathcal{Z}[\gamma_1](m_1, w_1) \le \mathcal{J}(m_2, w_2) - \mathcal{J}(m_1, w_1).
$$
\end{prop}

\begin{proof}
By adding $\mathcal{J}_1(m_2, w_2) - \mathcal{J}_1(m_1, w_1)$ to both sides of (\ref{LP23_lem24_ineq1}), we get the desired inequality.
\end{proof}

\begin{prop} \label{LP23_cor22}
Let $R > 0$. There exists a constant $C(R) > 0$ such that for any $(m, w) \in \mathcal{R}$ such that $\|m\|_{{W^{1,2,q}(Q)}} \le R$,
$$
\|m - \bar{m}\|_{{L^2(0, T; L^{\infty}(\T^d))}} \le C(R) \eps^{{\frac{1}{r}}}, \qquad
\|w - \bar{w}\|_{L^2(Q; \R^d)} \le C(R) \left(\sqrt{\eps}+ \eps^{{\frac{1}{r}}}\right),
$$
where $\eps = \mathcal{J}(m, w) - \mathcal{J}(\bar{m}, \bar{w}) \ge 0$.
\end{prop}

\begin{proof}
By applying Proposition \ref{LP23_cor25} with $(m_1, w_1) = (\bar{m}, \bar{w}), (m_2, w_2) = (m, w)$, we have $\sigma \le \eps$. Applying Proposition \ref{LP23_prop21} with $\widehat{\gamma} = \bar{\gamma}$, we complete the proof.
\end{proof}

With the above preliminaries, we are now able to present the following proofs.

\begin{proof}[Proof of Proposition \ref{LP23_lem6}]
\add{Let $m_{\star} \in W^{1, 2, q}(Q)$ with $m_\star(t,\cdot)\in  \mathcal{D}_1(\T^d)$ $(t\in [0,T])$ and set $\gamma = f(\cdot, \cdot, m_{\star}) \in \Gamma$.} 
The unique existence of a solution of \eqref{best} is a consequence of Lemmas \ref{LP23_lem13}--\ref{LP23_lem17}. 
Applying Lemma \ref{LP23_lem19} with $\widehat{\gamma} = \gamma$, $\widehat{m} := \bsym{m}[\gamma]$, and $\widehat{w} := \bsym{w}[\gamma]$, for any $(m, w) \in \mathcal{R}$, we have
$$
\mathcal{Z}[\gamma](m, w) - \mathcal{Z}[\gamma](\widehat{m}, \widehat{w}) \ge \frac{1}{2C_0} \iint_Q |v-\widehat{v}|^2 m \, dx dt \ge 0.
$$
Here, we have used that $m$ is non-negative by Lemma \ref{LP23_lem13}. Thus, $(\widehat{m}, \widehat{w})$ is a minimizer of $\mathcal{Z}[\gamma]$. For the uniqueness of the solution, we apply Proposition \ref{LP23_prop21}. Let
$$
\sigma := \mathcal{Z}[\gamma](m, w) - \mathcal{Z}[\gamma](\widehat{m}, \widehat{w}) \ge 0,
$$
then there exist constants $p \ge 2$ and $C > 0$ such that
$$
\|m - \widehat{m}\|_{{L^2(0, T; L^{\infty}(\T^d))}} \le C \sigma^{\frac{1}{r}}, \qquad \|w - \widehat{w}\|_{L^2(Q; \R^d)} \le C \left(\sqrt{\sigma} + \sigma^{\frac{1}{r}}\right).
$$
If $(m, w) \in \mathcal{R}$ is also a minimizer of $\mathcal{Z}[\gamma]$, $\sigma = 0$. Then, we have $(m, w) = (\widehat{m}, \widehat{w})$.
\end{proof}

\begin{proof}[Proof of Theorem \ref{LP23_thm7}]
\label{proofth3.6}
In view of Theorem \ref{LP23_prop23}, there exists a constant $C > 0$ such that $\gamma_k \in \Gamma_C$, $\|\bar{m}_k\|_{W^{1, 2, q}(Q)} \le C$ and $\|m_k\|_{W^{1, 2, q}(Q)}\le C$. Applying Proposition \ref{LP23_cor22} with $R = C$ and $(m, w) = (\bar{m}_k, \bar{w}_k)$, we obtain
%$$
%\|\bar{m}_k - \bar{m}\|_{L^2(0, T; L^{\infty}(\T^d))} + \|\bar{w}_k - \bar{w}\|_{L^2(Q; %\R^d)} \le C \eps_k^{\frac{1}{p}},
%$$
%which is 
\eqref{LP23_thm_ineq1}. 
%\item[(\ref{LP23_thm_ineq2})] 
Lemma \ref{LP23_lem18} with $(m_1, w_1) = (\bar{m}, \bar{w})$ and $(m_2, w_2) = (\bar{m}_k, \bar{w}_k)$ gives
\begin{gather*}
\|\gamma_k-\bar{\gamma}\|_{L^2(0, T; L^{\infty}(\T^d))} \le C \|\bar{m}_k - \bar{m}\|_{L^2(0, T; L^{\infty}(\T^d))}.
\end{gather*}
Combining this with (\ref{LP23_thm_ineq1}), we deduce (\ref{LP23_thm_ineq2}).

We apply Lemma \ref{LP23_lem15} with $R=C, \gamma_1 = \bar{\gamma}, \gamma_2 = \gamma_k$ to derive
$$
\|u_k - \bar{u}\|_{L^{\infty}(Q)} \le C \|\gamma_k-\bar{\gamma}\|_{L^2(0, T; L^{\infty}(\T^d))},
$$
which, together with (\ref{LP23_thm_ineq2}), implies (\ref{LP23_thm_ineq3}).

(\ref{LP23_thm_ineq4}) is a consequence of Proposition \ref{LP23_prop21} with $R=C, \widehat{\gamma} = \gamma_k$ and $(m, w) = (m_k, w_k)$.

Finally, the boundedness of $\eps_k$ follows from the boundedness of $\mathcal{J}(\bar{m}_k, \bar{w}_k)$, which is deduced from Theorem \ref{LP23_prop23}.%$$
\end{proof}

\section{Proofs of Theorems \ref{LP23_thm8-1} and \ref{LP23_thm8-2}}
\label{sec:proof-rate}

This section is devoted to the study of the explicit convergence rates of $\eps_k$ and $\sigma_k$. We begin by establishing a basic relationship between $\eps_k$ and $\sigma_k$. 

\begin{lemma} \label{LP23_lem26}
We have $\eps_k \le \sigma_k$ for all $k = 0, 1, 2, \dots$.
\end{lemma}

\begin{proof}
This lemma is taken from \cite[Lemmas 26]{LP23}. 
Particularly, choosing $(m_1, w_1) = (\bar{m}_k, \bar{w}_k)$ and $(m_2, w_2) = (m, w) \in \mathcal{R}$ in Corollary \ref{LP23_cor25} immediately gives the desired inequality. 
\end{proof}

\begin{prop} \label{LP23_lem27-1}
For any $\delta \in [0, 1]$ and $k = 0, 1, 2, \dots$, we have
\begin{equation}
\mathcal{J}(\bar{m}_k^{\delta}, \bar{w}_k^{\delta}) \le \mathcal{J}(\bar{m}_k, \bar{w}_k) - \delta \sigma_k + \delta^2 L_f \|m_k - \bar{m}_k\|_{L^2(0, T; L^{\infty}(\T^d))}^2,
\label{eq:62b}    
\end{equation}
where $(\bar{m}_k^{\delta}, \bar{w}_k^{\delta})$ is defined by \eqref{eq:mdel}. 
\end{prop}

\begin{proof}
Since $\mathcal{J}_1$ is convex,
\begin{equation} \label{LP23_lem27_eq1}
\mathcal{J}_1(\bar{m}_k^{\delta}, \bar{w}_k^{\delta}) - \mathcal{J}_1(\bar{m}_k, \bar{w}_k) \le \delta\left[\mathcal{J}_1(m_k, w_k) - \mathcal{J}_1(\bar{m}_k, \bar{w}_k)\right].
\end{equation}
Substituting $(m_1, w_1) = (\bar{m}_k, \bar{w}_k)$ and $(m_2, w_2) = (\bar{m}_k^{\delta}, \bar{w}_k^{\delta})$ for Lemma \ref{LP23_lem24}, \eqref{LP23_lem24_ineq1a}, we get
$$
\mathcal{J}_2(\bar{m}_k^{\delta}) - \mathcal{J}_2(\bar{m}_k) \le \delta \iint_Q \gamma_k (m_k - \bar{m}_k)\, dx dt + \delta^2 L_f D_k,
$$
where $D_k$ is define as \eqref{eq:s3-dk}. 
This, together with  \eqref{LP23_lem27_eq1}, implies
\begin{align*}
&\mathcal{J}(\bar{m}_k^{\delta}, \bar{w}_k^{\delta}) - \mathcal{J}(\bar{m}_k, \bar{w}_k)\\
&\le \delta\left[\mathcal{J}_1(m_k, w_k) - \mathcal{J}_1(\bar{m}_k, \bar{w}_k) + \iint_Q \gamma_k (m_k - \bar{m}_k)\, dx dt\right] + \delta^2 L_f D_k\\
&= - \delta \sigma_k + \delta^2 L_f D_k,
\end{align*}
where we have used the fact that $\sigma_k$ is expressed as 
%\begin{align*}
\[
\sigma_k %&= \mathcal{Z}[\gamma_k](\bar{m}_k, \bar{w}_k) - \mathcal{Z}[\gamma_k](m_k, w_k)\\
= \mathcal{J}_1(\bar{m}_k, \bar{w}_k) - \mathcal{J}_1(m_k, w_k) + \iint_Q \gamma_k (\bar{m}_k-m_k)\, dxdt.
\]
Noting that $
D_k \le  \|m_k - \bar{m}_k\|_{L^2(0, T; L^{\infty}(\T^d))}^2$, 
we derive \eqref{eq:62b}. 
\end{proof}

\begin{prop} 
\label{LP23_lem27-2}
There exists a constant $C>0$ such that 
\begin{align} 
\|m_k - \bar{m}_k\|_{L^2(0, T; L^{\infty}(\T^d))} 
& \le C\eps_k^{\frac{1}{r(r-1)}},\label{eq:mkmk}\\
\sigma_k &\le C \eps_k^{{\frac{2}{r(r-1)}}}
\label{LP23_lem27-2_ineq3}
\end{align}
for all $k = 0, 1, 2, \dots$. 
%\begin{equation} \label{LP23_lem27-2_ineq1}
%\|m_k - \bar{m}\|_{L^2(0, T; L^{\infty}(\T^d))}^{p-1} \le C \|\bar{m}_k - \bar{m}\|_{L^2(0, T; L^{\infty}(\T^d))}.
%\end{equation}
%Moreover, 
%\begin{equation} \label{LP23_lem27-2_ineq2}
%\|m_k - \bar{m}_k\|_{L^2(0, T; L^{\infty}(\T^d))}\le C \eps_k^{\frac{1}{p(p-1)}},
%\end{equation}
\end{prop}

\begin{proof}
In view of Proposition \ref{LP23_prop21}, we calculate as 
\begin{align*}
C \|m_k - \bar{m}\|_{L^2(0, T; L^{\infty}(\T^d))}^{r}
&\le \ \mathcal{Z}[\bar{\gamma}](m_k, w_k) - \mathcal{Z}[\bar{\gamma}](\bar{m}, \bar{w})\\
%=& \ \iint_Q\bar{\gamma}(m_k-\bar{m})\, dx dt + \mathcal{J}_1(m_k, w_k) - \mathcal{J}_1(\bar{m}, \bar{w})\\
&= \ \iint_Q \left(\bsym{\gamma}[\bar{m}]-\bsym{\gamma}[\bar{m}_k]\right)(m_k-\bar{m})\,dx dt \\
& \mbox{ }\quad + \underbrace{\iint_Q \bsym{\gamma}[\bar{m}_k](m_k-\bar{m})\,dx dt \add{+} \mathcal{J}_1(m_k, w_k) - \mathcal{J}_1(\bar{m}, \bar{w})}_{=\mathcal{Z}[\gamma_k](m_k, w_k) - \mathcal{Z}[\gamma_k](\bar{m}, \bar{w}) \le 0}\\
&\le \ \iint_Q \left(\bsym{\gamma}[\bar{m}]-\bsym{\gamma}[\bar{m}_k]\right)(m_k-\bar{m})\,dx dt.
\end{align*}
%Using the Lipschitz continuity of $f$ (assumption (H5)), we have
%$$
%\left|\bsym{\gamma}[\bar{m}]-\bsym{\gamma}[\bar{m}_k]\right| = \left|f(t, x, \bar{m}) - f(t, x, \bar{m}_k)\right| \le L_f \|\bar{m}(t) - \bar{m}_k(t)\|_{L^{\infty}(\T^d)},
%$$
%and hence
Hence, by the condition \textup{(f-L)}, 
\begin{align*}
\|m_k - \bar{m}\|_{L^2(0, T; L^{\infty}(\T^d))}^{r}
&\le   CL_f \iint_Q \|\bar{m}(t) - \bar{m}_k(t)\|_{L^{\infty}(\T^d)} |m_k(t, x) - \bar{m}(t, x)|\,dx dt \\
&\le C\|\bar{m} - \bar{m}_k\|_{L^2(0, T; L^{\infty}(\T^d))} \|m_k - \bar{m}\|_{L^2(0, T; L^{\infty}(\T^d))}.
\end{align*}
(We acknowledge that the proof of this inequality follows the ideas used in the proof of \cite[Lemma~4.8]{KW24}.)
%Dividing both sides by $\|m_k - \bar{m}\|_{L^2(0, T; L^{\infty}(\T^d))}$, we obtain (\ref{LP23_lem27-2_ineq1}):
%$$
%\|m_k - \bar{m}\|_{L^2(0, T; L^{\infty}(\T^d))}^{p-1} \le C \\|\bar{m}_k - \bar{m}\|_{L^2(0, T; L^{\infty}(\T^d))}.
%$$
This, together with Theorem \ref{LP23_thm7}, \eqref{LP23_thm_ineq1}, implies
$$
\|{m}_k - \bar{m}\|_{L^2(0, T; L^{\infty}(\T^d))} \le C 
\|\bar{m} - \bar{m}_k\|_{L^2(0, T; L^{\infty}(\T^d))}^{\frac{1}{r-1}}
\le 
C
\eps_k^{\frac{1}{r}\cdot\frac{1}{r-1}},
$$
and, hence, \eqref{eq:mkmk} is derived. 

Now, Lemma \ref{LP23_lem27-1} with $\delta = 1$ yields
\[
\mathcal{J}(\bar{m}, \bar{w}) \le \mathcal{J}(\bar{m}_k^{\delta}, \bar{w}_k^{\delta}) 
\le \mathcal{J}(\bar{m}_k, \bar{w}_k) - \sigma_k + C \eps_k^{\frac{2}{r(r-1)}},
\]
%\begin{align*}
%\mathcal{J}(\bar{m}, \bar{w}) \le \mathcal{J}(\bar{m}_k^{\delta}, \bar{w}_k^{\delta}) &\le \mathcal{J}(\bar{m}_k, \bar{w}_k) - \sigma_k + L_f \|m_k - \bar{m}_k\|_{L^2(0, T; L^{\infty}(\T^d))}^2\\
%&\le \mathcal{J}(\bar{m}_k, \bar{w}_k) - \sigma_k + C \eps_k^{\frac{2}{p(p-1)}}.
%\end{align*}
which leads to \eqref{LP23_lem27-2_ineq3}. 
\end{proof}

\begin{prop} \label{KW22_prop4.2}
Assume that $\delta_k$ is given by \textup{(S1)}. Let $q_k$ be a sequence satisfying
$$
0 < q_k \le c \min{\left\{\frac{\tau (1-c) \sigma_k}{L_f \|m_k - \bar{m}_k\|_{L^2(0, T; L^{\infty}(\T^d))}^2}, 1\right\}}
$$
for all $k = 0, 1, 2, \dots$. Then, we have $q_k\le c\delta_k$, and $\eps_{k+1} \le (1-q_k) \eps_k$.    
\end{prop}

\begin{proof}
It basically follows \cite[Section~4]{KW24}. 
From $\eps_k \le \sigma_k$ and the QAG condition \eqref{QAGcond},
$$
c \delta_k \eps_k \le c \delta_k \sigma_k \le \mathcal{J}(\bar{m}_k, \bar{w}_k) - \mathcal{J}(\bar{m}_{k+1}, \bar{w}_{k+1}) = \eps_k - \eps_{k+1}.
$$
By rearranging, we have
\begin{equation} \label{KW22_prop4.2_ineq0}
\eps_{k+1} \le (1 - c\delta_k) \eps_k.
\end{equation}
If $\delta_k = 1$, then
\begin{equation} \label{KW22_prop4.2_ineq1}
\eps_{k+1} \le (1 - c) \eps_k.
\end{equation}
We now let $\delta_k \in (0, 1)$ and note $\delta_k\le\tau$. 

%As was stated in Remark \ref{rem:s1}, 

Then, there exists $\widehat{\delta}_k \in [\delta_k, \delta_k/\tau)$ such that
\begin{equation} \label{KW22_prop4.2_eq2}
c = \frac{\mathcal{J}(\bar{m}_k, \bar{w}_k) - \mathcal{J}\left(\bar{m}_k^{\widehat{\delta}_k}, \bar{w}_k^{\widehat{\delta}_k}\right)}{\widehat{\delta}_k \sigma_k}.
\end{equation}
%To show this, note that the function
%$$
%W_k(\delta) := \frac{\mathcal{J}(\bar{m}_k, \bar{w}_k) - \mathcal{J}(\bar{m}_k^{\delta}, \bar{w}_k^{\delta})}{\delta \sigma_k}\qquad (0 < \delta < 1)
%$$
%is continuous (see \cite[Lemma 3.2]{KW24} for details).
{To show this, consider the function
$$
W_k(\delta) := \frac{\mathcal{J}(\bar{m}_k, \bar{w}_k) - \mathcal{J}(\bar{m}_k^{\delta}, \bar{w}_k^{\delta})}{\delta \sigma_k}, 
$$
which is contious in $(0,1]$ (see Remark \ref{rem:Wk} below). }
Since
\begin{align}
W_k(\delta_k) &\ge \frac{c \delta_k \sigma_k}{\delta_k \sigma_k} = c, \tag{$\delta_k$ satisfies (\ref{QAGcond})}\\
W_k\left(\frac{\delta_k}{\tau}\right) &< \frac{c \delta_k \sigma_k}{\delta_k \sigma_k} = c, \tag{$\delta_k/\tau$ does not satisfy (\ref{QAGcond})}
\end{align}
we can take $\widehat{\delta}_k \in [\delta_k, \delta/\tau)$ such that $W_k(\widehat{\delta}_k) = c$ by the intermediate value theorem.

Now, by applying Lemma \ref{LP23_lem27-1} to (\ref{KW22_prop4.2_eq2}), we have
%\begin{align*}
\[
c \ge \frac{\widehat{\delta}_k \sigma_k - L_f \widehat{\delta}_k^2 \|m_k - \bar{m}_k\|_{L^2(0, T; L^{\infty}(\T^d))}^2}{\widehat{\delta}_k \sigma_k}.
\]
%&= 1 - \frac{L_f \widehat{\delta}_k}{\sigma_k} \|m_k - \bar{m}_k\|_{L^2(0, T; L^{\infty}(\T^d))}^2\\
%&\ge 1 - \frac{L_f \delta_k}{\tau \sigma_k} \|m_k - \bar{m}_k\|_{L^2(0, T; L^{\infty}(\T^d))}^2 ,
%\end{align*}
Therefore, by  $\widehat{\delta}_k<\delta_k/\tau$, we obtain 
$$
1 > \delta_k \ge \frac{\tau(1-c) \sigma_k}{L_f \|m_k - \bar{m}_k\|_{L^2(0, T; L^{\infty}(\T^d))}^2},
$$
which implies that $q_k\le c\delta_k$ holds. 
Using the inequality (\ref{KW22_prop4.2_ineq0}), we have
$$
\eps_{k+1} \le \left(1-\frac{\tau c (1-c) \sigma_k}{L_f \|m_k - \bar{m}_k\|_{L^2(0, T; L^{\infty}(\T^d))}^2}\right) \eps_k.
$$
By combining this inequality with the inequality (\ref{KW22_prop4.2_ineq1}), we obtain the second inequality to be proved. \end{proof}

\begin{rem}
    \label{rem:Wk}
Since $\vhi(\delta) := \mathcal{J}(\bar{m}_k^{\delta}, \bar{w}_k^{\delta})$ is convex and lower semi-continuous on $[0, 1]$, it is continuous on $(0, 1)$ (see \cite[Proposition 2.5]{ET99} for example). Moreover, we infer that $\vhi$ is upper semi-continuous at $\delta =1$, because
$$
\limsup_{\delta \nearrow 1} \vhi(\delta) \le \limsup_{\delta \nearrow 1} [(1-\delta) \vhi(0) + \delta \vhi(1)] = \vhi(1).
$$
Thus, $\vhi$ is continuous at $\delta = 1$ and, consequently, the function $W_k$ is also continuous on $(0, 1]$.
\end{rem}

\begin{lemma}
    \label{KW22_prop4.10}
Let $\alpha \in [1/2, 1), \beta > 0$, and $\gamma \in (0, 1)$. If $h_k$ is a non-negative sequence such that
$$
h_{k+1} \le h_k \max{\left\{1-\beta h_k^\gamma, \alpha \right\}} \qquad (k=0, 1, 2, \dots),
$$
then there exist positive constants $N, M$ such that $h_k \le \dfrac{M}{(k+N)^{\frac{1}{\gamma}}}$. Therein, $N$ and $M$ are given by
$$
N = \frac{2-\alpha^{-\gamma}}{\alpha^{-\gamma}-1}, \qquad M = \max{\left\{h_0 N^{\frac{1}{\gamma}}, \frac{1}{\alpha \left\{\left[\gamma-(1-\gamma)(2^\gamma-1)\right] \beta\right\}^{\frac{1}{\gamma}}}\right\}}.
$$
\end{lemma}

\begin{proof}
This lemma is taken from \cite[Proposition 4.10]{KW24}.
The proof is not provided there, but it follows from an elementary induction argument, similar to the proof of \cite[Theorem 1]{XY18}, which deals with the case $\alpha=1/2$ is considered. 
\end{proof}

We now state the following proofs.  

\begin{proof}[Proof of Theorem \ref{LP23_thm8-1}]
\label{proofthe3.7}
The case $d = 1$ is the same as the proof of \cite[Theorem 8]{LP23}. We consider only the case $d \ge 2$. Recall that we have set $\displaystyle{\rho = 1-\frac{2}{r(r-1)}}$. 

\noindent 1) \emph{ Let $\delta_k$ be defined by \textup{(S1)}.} We define 
\begin{equation}
\label{eq:qk}
q_k = c \min{\left\{\frac{\tau (1-c)}{L_fC_1} \eps_k^\rho, ~1\right\}},
\end{equation}
where we denote the constant $C$ appearing in \eqref{eq:mkmk} by $C_1$. Since, by Lemma \ref{LP23_lem26} and \eqref{eq:mkmk}, 
\[
\frac{\tau (1-c)}{L_fC_1} \eps_k^{1-\frac{2}{r(r-1)}}\le \frac{\tau (1-c) \sigma_k}{L_f \|m_k - \bar{m}_k\|_{L^2(0, T; L^{\infty}(\T^d))}^2},
\]
the sequence $q_k$ fulfills the assumption of Proposition \ref{KW22_prop4.2}. Therefore, we have
$$
\eps_{k+1} \le (1-q_k) \eps_k = \eps_k \max{\left\{1-\frac{\tau c(1-c)}{L_fC_1} \eps_k^\rho, ~1-c\right\}}.
$$
Since $1/2\le 1-c<1$ and $\tau c(1-c)/(L_fC_1)>0$, we can apply Lemma \ref{KW22_prop4.10} and obtain \eqref{eq:rate-d2}

\smallskip

\noindent 2) \emph{Let $\delta_k$ be defined by \textup{(S2)}.} 
We define $q_k$ by \eqref{eq:qk} again. 
In this case, Proposition \ref{KW22_prop4.2} is not available. \add{But, if the inequality $\eps_{k+1} \le (1-q_k) \eps_k$ holds, we can conclude that \eqref{eq:rate-d2} holds by the same way as 1) above.} We prove that this inequality holds by arguing by contradiction. Suppose that $\eps_{k+1} > (1-q_k) \eps_k$ holds. Then, 
\begin{align*}
\mathcal{J}\left(\bar{m}_k^{\delta_k}, \bar{w}_k^{\delta_k}\right) &> (1-q_k) \mathcal{J}(\bar{m}_k, \bar{w}_k) + q_k \mathcal{J}(\bar{m}, \bar{w})\\
&= \mathcal{J}(\bar{m}_k, \bar{w}_k) - q_k \left(\mathcal{J}(\bar{m}_k, \bar{w}_k)-\mathcal{J}(\bar{m}, \bar{w})\right)\\
&\ge \mathcal{J}(\bar{m}_k, \bar{w}_k) - q_k \sigma_k.
\end{align*}
Let $\delta_k^{\mathrm{QAG}}$ be the step-size defined by (S1). We have derived $q_k \le c \delta_k^{\textup{QAG}}$ in Lemma \ref{KW22_prop4.2}. Thanks to \eqref{QAGcond}, we observe
\begin{equation*}
\mathcal{J}\left(\bar{m}_k^{\delta_k}, \bar{w}_k^{\delta_k}\right) > \mathcal{J}(\bar{m}_k, \bar{w}_k) - c \delta_k^{\textup{QAG}} \sigma_k
\ge \mathcal{J}\left(\bar{m}_k^{\delta_k^{\textup{QAG}}}, \bar{w}_k^{\delta_k^{\textup{QAG}}}\right).
\end{equation*}
This contradicts the fact that $\delta_k$ minimizes $\mathcal{J}(\bar{m}_k^{\delta}, \bar{w}_k^{\delta})$. Therefore, we get $\eps_{k+1} \le (1-q_k) \eps_k$. 

\smallskip

\noindent 3) \emph{ Let $\delta_k$ be defined by \textup{(S3)}.} Set $a_k=L_f D_k$. If $\sigma_k \ge 2 a_k$, then we have $\delta_k = 1$. By Lemma \ref{LP23_lem27-1}, 
$$
\eps_{k+1} \le \eps_k - \sigma_k + a_k \le \eps_k - \frac{1}{2} \sigma_k \le \frac{1}{2} \eps_k.
$$
In the case of $\sigma_k < 2 a_k$, we have $\delta_k = \sigma_k/2a_k$. 
Note that $a_k \le C \eps_k^{\frac{2}{r(r-1)}}$. 
We apply Lemma \ref{LP23_lem27-1} again and deduce
%\begin{align*}
\[
\eps_{k+1} 
\le \eps_k - \frac{\sigma_k^2}{2a_k} + \frac{\sigma_k^2}{4 a_k} 
%= \eps_k - \frac{\sigma_k^2}{4a_k} 
%\le \eps_k - \frac{1}{4C} \eps_k^{2-\frac{2}{r(r-1)}}\\
\le \eps_k \left(1-\frac{1}{4C} \eps_k^\rho\right).
%\end{align*}
\]
Therefore, we obtain
$$
\eps_{k+1} \le \eps_k \max{\left\{1-\frac{1}{4C} \eps_k^\rho, ~\frac{1}{2}\right\}}.
$$
Application of Lemma \ref{KW22_prop4.10} with $\alpha = 1/2, \beta = 1/(4C)$ leads to the desired inequality.
\end{proof}

%\begin{note}
%For the case of the QAG condition, $\alpha = 1 - c \in [1/2, 1)$ is required for using Lemma \ref{KW22_prop4.10}. Hence, $c$ is taken as $c \in (0, 1/2]$. Incidentally, we can take $c$ as $c \in (0, 1)$ for the case $d = 1$, since we do not have to rely on Lemma \ref{KW22_prop4.10}.
%\end{note}

\begin{proof}[Proof of Theorem \ref{LP23_thm8-2}]
\label{proofthe3.8}
We apply Proposition~\ref{LP23_lem27-1} with $\delta=\delta_k$ to obtain
\begin{equation*}
\varepsilon_{k+1}
\le \varepsilon_k
- \delta_k \sigma_k
+ \delta_k^2 L_f \|m_k - \bar{m}_k\|_{L^2(0,T;L^\infty(\T^d))}^2.
\end{equation*}
Together with Lemma~\ref{LP23_lem26} and Proposition~\ref{LP23_lem27-2}, this implies \eqref{thm8-2_eq}.

We prove the case $d = 1$. To this end, we will first derive, for $\delta_k \in [0, 1]$,
\begin{equation}
\label{eq:th83a}
\eps_k \le \eps_0 \exp{\left[\sum_{j=0}^{k-1} (C_{\ast} \delta_j^2 - \delta_j)\right]}.
\end{equation}
\add{Taking $r = 2$} in \eqref{thm8-2_eq} leads to 
$\eps_{k+1} \le (1-\delta_k + C_{\ast} \delta_k^2) \eps_k$. 
Consequently, 
$$
\eps_k \le \eps_0 \prod_{j=0}^{k-1} (1-\delta_j+C_{\ast} \delta_j^2) = \eps_0 \exp{\left[\sum_{j=0}^{k-1} \log{(1-\delta_j+C_{\ast} \delta_j^2)}\right]},
$$
which, together with an elementary inequality $\log{s} < s-1$ for $s > 0$, implies \eqref{eq:th83a}. 
%we have
%$$
%\eps_k \le \eps_0 \exp{\left[\sum_{j=0}^{k-1} (C_{\ast} \delta_j^2 - \delta_j)\right]}.
%$$
Using the explicit expression of $\delta_k$, we estimate the right-hand side of \eqref{eq:th83a} as follows:
\begin{align*}
\sum_{j=0}^{k-1} \left(\frac{k_2}{j+k_1}\right)^2 &= k_2^2 \left(\frac{1}{k_1^2} + \sum_{j=1}^{k-1} \int_{j-1}^j \frac{ds}{(j+k_1)^2}\right)\\
&\le k_2^2 \left(\frac{1}{k_1^2} + \int_0^{\infty} \frac{ds}{(s+k_1)^2}\right)
= \frac{k_2^2}{k_1^2}(k_1+1),\\
\sum_{j=0}^{k-1} \frac{k_2}{j+k_1} &= \sum_{j=0}^{k-1} \int_j^{j+1} \frac{k_2}{j+k_1}\, ds\\
&\ge \int_0^k \frac{k_2}{s+k_1}\,ds = k_2 \log{\frac{k+k_1}{k_1}}.
\end{align*}
By rearranging, we \add{complete} the proof for the case $d = 1$.

\medskip

We proceed to the proof of the case $d \ge 2$. It will be done by induction. In order for the case $k=0$ to be valid, it must be that $N \ge \eps_0 k_1^s$. In general, assuming it holds for $k$, we consider the case $k+1$. Using \eqref{thm8-2_eq}, we estimate as 
\begin{align*}
\eps_{k+1} &\le \left(1-\frac{k_2}{k+k_1}\right) \frac{N}{(k+k_1)^s} + C_{\ast} \left(\frac{k_2}{k+k_1}\right)^2 \frac{N^{1-\rho}}{(k+k_1)^{(1-\rho)s}}\\
&= \left[\left(1-\frac{k_2}{k+k_1}\right)\left(\frac{k+k_1+1}{k+k_1}\right)^s + \frac{C_{\ast}}{N^\rho} \left(\frac{k_2}{k+k_1}\right)^2 \frac{(k+k_1+1)^s}{(k+k_1)^{(1-\rho)s}}\right] \frac{N}{(k+k_1+1)^s}.
\end{align*}
If $s \le 1/\rho$, we have
$
(k+k_1)(k+k_1)^{(1-\rho)s} 
%= (k+k_1)^{1-\rho s} (k+k_1)^s 
\ge (k+k_1)^s$, 
and thus,
\begin{align*}
\eps_{k+1} &\le \left[\left(1-\frac{k_2}{k+k_1}\right)\left(\frac{k+k_1+1}{k+k_1}\right)^s + \frac{C_{\ast}}{N^\rho} \frac{k_2^2}{k+k_1} \frac{(k+k_1+1)^s}{(k+k_1)^s}\right]\frac{N}{(k+k_1+1)^s}\\
&\le \underbrace{\left(1-\frac{k_2 - C_{\ast} N^{-\rho} k_2^2}{k+k_1}\right) \left(1 + \frac{1}{k+k_1}\right)^s}_{=: \eta(k)} \frac{N}{(k+k_1+1)^s}.
\end{align*}
Set $a_N := k_2 - C_{\ast} N^{-\rho} k_2^2$. We have
$$
\frac{d}{dk}\eta(k) = \frac{1}{(k+k_1)^2}\left[\frac{a_N(s+1)}{k+k_1} + (a_N - s)\right] \left(1 + \frac{1}{k+k_1}\right)^{s-1}.
$$
If $a_N \ge s$, that is, if $s < k_2$ and $N \ge \left[C_{\ast} k_2^2/(k_2-s)\right]^{1/\rho}$, then $d\eta/dk > 0$. Thus, since $\eta$ is increasing and $\lim\limits_{k \to \infty} \eta(k) = 1$, we have $\eta(k) \le 1$. Consequently,
$$
\eps_{k+1} \le \frac{N}{(k+k_1+1)^s}.
$$
Therefore, by choosing $s$ and $N$ as in the statement, we obtain $\eps_k \le N/(k+k_1)^s$. This completes the proof of Theorem  \ref{LP23_thm8-2}. 
\end{proof}

%\begin{rem} \label{rem_thm8-2}
%The convergence speed of $\eps_k$ seems to be faster as $k_2$ increases. However, since $k_1$ ($\ge k_2$) also increases, the coefficient $N$ grows rapidly as a result. Therefore, increasing $k_1$ and $k_2$ would not improve the convergence speed. We numerically verify this in section \ref{sec_numerical_experiment}.
%\end{rem}

\section{Proof of Theorem \ref{thm_MFG_regularity}} \label{sec_MFG_regularity}

We prove the existence of the smooth solutions of \eqref{MFG_eq}. For the proof, we follow the method of the proof of \cite[Theorem 1]{B21}, which is an application of the Leray--Schauder theorem (\cite[Theorem 11.6]{GT01} for example):

\begin{lemma}
\label{la:ls}
For a Banach space $X$, let $\mathcal{T} : X \times [0, 1] \to X$ be a continuous and compact mapping that satisfies:  
\begin{itemize}
\item[\textup{(LS1)}] There exists $\xi_0 \in X$ such that for any $\xi \in X$, one has $\mathcal{T}(\xi, 0) = \xi_0$; 
\item[\textup{(LS2)}] There exists a constant $C > 0$ such that for any $(\xi, \tau) \in X \times [0, 1]$ such that $\mathcal{T}(\xi, \tau) = \xi$, one has $\|\xi\|_X \le C$.
\end{itemize}
Then, there exists $\xi \in X$ such that $\mathcal{T}(\xi, 1) = \xi$. That is, $\mathcal{T}(\cdot, 1)$ has a fixed point.
\end{lemma}
%\end{theorem}

To apply this lemma to the MFG system \eqref{MFG_eq}, let $X := W^{1, 2, q}(Q) \times W^{1, 2, q}(Q)$ and define the mapping $\mathcal{T}(u, m, \tau):=(u_{\tau}, m_{\tau})$ for $(u, m, \tau) \in X \times [0, 1]$, where $(u_{\tau}, m_{\tau})$ denotes the unique solution to the following parametrized MFG system:
\begin{subequations}
\label{MFG_eq_tau}
\begin{alignat}{2}
-\del_t u_{\tau} - \nu \Delta{u}_{\tau} + \tau H(t, x, \nabla{u}) &= \tau f(t, x, \rho[m(t)]) &\quad& \mbox{in } Q, \label{MFG_eq_taua}\\ 
u_{\tau}(T,x) &= \tau g(x) && \mbox{in } \T^d,\label{MFG_eq_taub} \\
v_{\tau}(t, x) &= - \nabla_p H(t, x, \nabla{u}_{\tau}) &&\mbox{in }Q, \label{MFG_eq_tauc}\\
 \del_t m_{\tau} - \nu \Delta{m}_{\tau} + \tau \nabla \cdot {\left(\rho(m) v_{\tau}\right)} &= 0 & &\mbox{in } Q, \label{MFG_eq_taud}\\ 
 m_{\tau}(0,x) &= m_0(x) && \mbox{in } \T^d. \label{MFG_eq_taue}
\end{alignat}
%\end{dcases} \tag{$\text{MFG}_{\tau}$}
%\end{equation}
\end{subequations}
Therein, $\rho : \mathcal{C}(Q) \to \mathcal{C}(0, T; \mathcal{D}_1(\T^d))$ is the projection operator by \cite{B23} defined as 
$$
\rho[m](t, x) := 1 + \frac{m_+(t, x) - M_+}{\max{\{1, M_+\}}} \qquad \left(M_+ := \int_{\T^d} m_+(t, y)\, dy, m_+ := \max{\{m, 0\}}\right).
$$
Actually, the mapping $\mathcal{T}$ is well-defined in view of Lemmas \ref{LP23_lem13}, \ref{LP23_lem15}, \ref{LP23_lem16}, and \ref{LP23_lem17}. 
As is stated in \add{\cite[the proof of Theorem 1]{B21}}, one knows:

\begin{itemize}
\item $\rho[m] = m$ for $m \in \mathcal{D}_1(\T^d)$;
\item $\|\rho[m_1] - \rho[m_2]\|_{L^{\infty}(\T^d)} \le L_{\rho} \|m_2 - m_1\|_{L^{\infty}(\T^d)}$ for $m_1,m_2 \in L^{\infty}(\T^d)$ with a constant $L_\rho>0$; 
\item For any $\alpha \in (0, 1)$, there is a constant $C_\alpha>0$ such that $\|\rho[m]\|_{\mathcal{C}^{\alpha}(\T^d)} \le C_\alpha \|m\|_{\mathcal{C}^{\alpha}(\T^d)}$ for $m \in \mathcal{C}^{\alpha}(\T^d)$.
\end{itemize}

\begin{proof}[Proof of Theorem \ref{thm_MFG_regularity}]
The proof consists of the following three parts:
\begin{enumerate}
\item[1)]  showing that the mapping 
$\mathcal{T}(\cdot, \cdot, 1)$ has a fixed point $(\bar{u}, \bar{m}) \in X$. Then, $(\bar{u}, \bar{m})$ is a solution of \eqref{MFG_eq}; 
\item[2)] proving the uniqueness of the solution $(\bar{u}, \bar{m})$;
\item[3)] showing that the solution $(\bar{u}, \bar{m})$ is a classical one in the sense that \eqref{eq:classical} hold.
\end{enumerate}
However, as is well known, the uniqueness result in 2) follows from the monotonicity condition \textup{(f–M)}; see, for instance, \cite{car12,ll06b,LL07}.
Moreover, 3) is obtained as a direct application of Remark \ref{B21_lem12} and Lemma \ref{B21_thm7} (the variable $t$ should be understood as $-t$). Thus, the only part that remains to be verified here is 1). 

We proceed by checking one by one that the assumptions of Lemma~\ref{la:ls} are satisfied.

The proof of continuity and compactness of $\mathcal{T}$ is the same as the proof of \cite[Theorem 1]{B23}, so we skip it. Since $\mathcal{T}(u, m, 0) = (0, 0)$, condition (LS1) immediately verified. 

Let us check condition (LS2).  
Given $\tau \in (0, 1]$, let $(u_{\tau}, m_{\tau}) \in X$ be a fixed point of $\mathcal{T}(\cdot, \cdot, \tau)$. Below $C$ denotes generic positive constants independent of $\tau$.
The proof relies on the following two key inequalities:
\begin{subequations}
\label{eq:bdd-duk21}
\begin{align}
\|u_{\tau}\|_{L^{\infty}(Q)} &\le C, \label{eq:bdd-duk20}\\
\|\nabla u_{\tau}\|_{L^{\infty}(Q)} &\le C. \label{eq:bdd-duk2}
\end{align}
\end{subequations}
Since the proof of \eqref{eq:bdd-duk20} is essentially identical to the first step of the proof of Proposition \ref{prop:du_k}, we omit the details.
We defer the proof of \eqref{eq:bdd-duk2} and proceed with the proof of the theorem.

In view of Lemma \ref{B21_thm4}, we infer $\|u_{\tau}\|_{W^{1, 2, q}(Q)} \le C$.
Moreover, since
$$
v_{\tau} = -\nabla_p H(\nabla u_{\tau}) = -\nabla u_{\tau} + h, \qquad 
\nabla \cdot v_{\tau} = - \Delta u_{\tau} + \nabla \cdot h,
$$
it follows that $\|v_{\tau}\|_{L^q(Q)}\le C$ and $\|\nabla \cdot v_{\tau}\|_{L^q(Q)}\le C$. Hence, applying Lemma \ref{B21_thm4} to \eqref{MFG_eq_tauc}, \eqref{MFG_eq_taud} and \eqref{MFG_eq_taue},  
we obtain $\|m_{\tau}\|_{W^{1, 2, q}(Q)} \le C$. Thus, condition (LS2) is verified. 
Therefore, we can apply Lemma \ref{la:ls}, to conclude that $\mathcal{T}(\cdot, \cdot, 1)$ has a fixed point in $X$.

\smallskip

What is left is to prove \eqref{eq:bdd-duk2}. 
We introduce $\phi_\tau$ by the Cole--Hopf transformation $\phi_{\tau} = \tau \exp{\left(- \frac{\tau u_{\tau}}{2 \nu}\right)}$,  then the function $\phi_{\tau}$ solves
\begin{equation} 
\label{parametrizedHJB}
\del_t \phi_{\tau} + \nu \Delta \phi_{\tau} + \tau h \cdot \nabla \phi_{\tau} = \frac{\tau^2}{2 \nu} f_{\tau} \phi_{\tau} \quad \mbox{in } Q, \quad
\phi_{\tau}(T,\cdot) = \tau \exp{\left(-\frac{\tau^2 g}{2 \nu}\right)} \quad  \mbox{on } \T^d.
\end{equation}
Additionally, since
\begin{align*}
\del_i \phi_{\tau}(T,\cdot) &= - \frac{\tau^3 \del_i g}{2 \nu} \exp{\left(-\frac{\tau^2 g}{2 \nu}\right)}, \\
\del_j \del_i \phi_{\tau}(T,\cdot) &= \left[- \frac{\tau^3 \del_j \del_i g}{2 \nu} + \frac{\tau^5 (\del_j g) (\del_i g)}{(2 \nu)^2}\right]\exp{\left(-\frac{\tau^2 g}{2 \nu}\right)},
\end{align*}
we have 
\begin{equation}
    \label{eq:LS-proof}
\|\phi_{\tau}(T,\cdot)\|_{W^{1,2,q}(Q)}\le \tau C,\qquad 
\|\Delta \phi_{\tau}(T,\cdot)\|_{L^{q}(Q)}\le \tau^3C. 
\end{equation}

The relationship between $u_\tau$ snd $\phi_\tau$ is written as
$$
\nabla u_{\tau} = -\frac{2 \nu}{\tau^2 \exp{\left(-\frac{\tau u_{\tau}}{2 \nu}\right)}} \nabla \phi_{\tau}.
$$
Therefore, it suffices to show that $\|\nabla \phi_{\tau}\|_{L^{\infty}(Q)} \le \tau^2C$. The proof of this is divided into three steps below. 

\medskip

\noindent \emph{Step 1. To show $\|\nabla \phi_{\tau}\|_{L^{\infty}(Q)} \le C$.} 
We derive rough estimations that reflect the dependence on $\tau$: 
$$
\|\tau h\|_{L^q(Q; \R^d)} \le C, \qquad 
\left\|\frac{\tau^2}{2 \nu} f_{\tau} \phi_{\tau}\right\|_{L^q(Q)} \le C,
$$
and 
$$
\|\phi_{\tau}(T,\cdot)\|_{W^{2-2/q, q}(\T^d)} \le C \|\phi_{\tau}(T,\cdot)\|_{W^{2, q}(\T^d)} \le C. 
$$
We apply Lemma \ref{B21_thm4} and obtain $\|\phi_{\tau}\|_{W^{1, 2, q}(Q)} \le C$ and hence $\|\nabla \phi_{\tau}\|_{L^{\infty}(Q)} \le C$.

\medskip

\noindent \emph{Step 2. To show $\|\nabla \phi_{\tau}\|_{L^{\infty}(Q)}\le \tau C$.} 
We already know \eqref{eq:LS-proof}. Using the result of Step 1, we have
$$
\left\|\frac{\tau^2}{2 \nu} f_{\tau} \phi_{\tau} - \tau h \cdot \nabla \phi_{\tau}\right\|_{L^q(Q)} \le \tau C.
$$
Application of Lemma \ref{B21_thm4} gives $\|\phi_{\tau}\|_{W^{1, 2, q}(Q)} \le \tau C$ and hence $\|\nabla \phi_{\tau}\|_{L^{\infty}(Q; \R^d)} \le \tau C$.

\medskip

\noindent \emph{Step 3. To show $\|\nabla \phi_{\tau}\|_{L^{\infty}(Q)} \le \tau^2 C$.} Set $\vhi_{\tau} := \phi_{\tau} - \phi_{\tau}(T,\cdot)$. Then $\vhi_{\tau}$ is a solution of 
$$
\del_t \vhi_{\tau} + \nu \Delta \vhi_{\tau} = F_{\tau} \quad \mbox{in } Q, \qquad \vhi_{\tau}(T,\cdot) = 0 \quad \mbox{on } \T^d,
$$
where
$$
F_{\tau} = \frac{\tau^2}{2 \nu} f_{\tau} \phi_{\tau} - \tau h \cdot \nabla \phi_{\tau} - \nu \left[-\frac{\tau^3}{2 \nu} \Delta g + \frac{\tau^5}{4 \nu^2} |\nabla g|^2\right]\exp{\left(-\frac{\tau^2 g}{2 \nu}\right)}.
$$
We have
$$
\left\|\frac{\tau^2}{2 \nu} f_{\tau} \phi_{\tau} - \tau h \cdot \nabla \phi_{\tau}\right\|_{L^q(Q)} \le \tau^2 C.
%, \qquad \|\nu \Delta \phi_{\tau}(T)\|_{L^q(Q)} \le \tau^2 C.
$$
Applying Lemma \ref{B21_thm4} again, we obtain $\|\vhi_{\tau}\|_{W^{1, 2, q}(Q)} \le \tau^2 C$. Thus,
\begin{align*}
\|\nabla \phi_{\tau}\|_{L^{\infty}(Q; \R^d)} &\le \|\nabla \phi_{\tau} - \nabla \phi_{\tau}(T,\cdot)\|_{L^{\infty}(Q; \R^d)} + \|\nabla \phi_{\tau}(T,\cdot)\|_{L^{\infty}(Q; \R^d)}\\
&\le C \|\vhi_{\tau}\|_{W^{1, 2, q}(Q)} + \left\|-\frac{\tau^3}{2 \nu} 
\exp{\left(-\frac{\tau^2 g}{2 \nu}\right)} \nabla g\right\|_{L^{\infty}(Q)}\le \tau^2 C, 
\end{align*}
which completes the proof. 
\end{proof}

\add{
\begin{rem}
\label{rem:extension_thm2.8}
We briefly comment on whether Theorem 2.8 can be extended to the class of Hamiltonians satisfying (H$'$). The difficulty in obtaining a uniform bound for $\|\nabla u_\tau\|_{L^{\infty}(Q;\R^d)}$ is closely related to the one discussed in Proposition \ref{prop:du_k}, and the $W^{1,\infty}(Q)$-estimate from \cite{CG20}, recalled in Remark \ref{rem:rem-h}, may seem useful for this purpose. However, even if one uses such an estimate, the dependence of the resulting bound on the parameter $\tau$ is not explicit. Thus, it is unclear whether the required bound can be obtained uniformly in $\tau$.
\end{rem}
}

\section{Numerical experiments}
\label{sec_numerical_experiment}

\subsection{Discretization method}
\label{sec_numerical_experiment_1}

In this section, we present results from numerical experiments to confirm the validity of our theoretical results. First, we explain the discretization method employed for the GCG method \eqref{GCG}. Recall that  $H$ is given as \textup{(H)}. 
Thanks to the Cole--Hopf transformation $\phi_k = \exp{(-u_k/2 \nu)}, \psi_k = m_k/\phi_k$, the GCG method \eqref{GCG} is equivalently written as
\begin{subequations} 
\label{CH_MFG_eq}
\begin{alignat}{2}
\gamma_k&=f(\cdot,\cdot,\bar{m}_k)&&  \mbox{in } Q, \\
\del_t \phi_k + \nu \Delta{\phi_k} + h\cdot \nabla \phi_k &= \frac{1}{2 \nu} \gamma_k \phi_k &\quad& \mbox{in } Q, \\
\phi_k(T,\cdot) &= \exp{\left(-\frac{g}{2 \nu}\right)}&& \mbox{on } \mathbb{T}^d, \\
\del_t \psi_k - \nu \Delta{\psi_k} + \nabla \cdot (h \psi_k) &= -\frac{1}{2 \nu} \gamma_k \psi_k && \mbox{in } Q, \\
\psi_k(0,\cdot) &= \frac{m_0}{\phi_k(0)} && \mbox{on } \T^d.
\end{alignat}
and
\begin{equation}
m_k=\phi_k\psi_k,\quad 
\bar{m}_{k+1}=(1-\delta_k) \bar{m}_k+\delta_k m_k.% && 
\end{equation}
%\tag{CH-MFG[$\gamma_k$]}
\end{subequations}
We then solve \eqref{CH_MFG_eq} using the finite difference method proposed in \cite{I23}. 
We evaluate the integrals arising in the exploitability $\sigma_k$ by the rectangle formula (trapezoid formula)
$$
\int_{x_0}^{x_n} u(x)\, dx \approx \frac{1}{n} \sum_{i=0}^{n-1} u(x_i).
$$

We then describe how we implement the step-size selection methods (S1), (S3), (S3), and (S4). 
In the QAG step-size (S1), we set $c = 0.25, \tau = 0.75$. 
For the optimal step-size (S2), following \cite{LP23}, we make use of the golden-section method as is given in Algorithm \ref{alg_goldensect}. 
It is effective for optimizing one-dimensional convex functions and does not require the differentiability of objective functions for solving the optimization problem \eqref{optimize_J(delta)}. 
In fact, this is an algorithm that narrows the search range $[a, d]$, and when the size of the search range becomes smaller than a tolerance $\kappa$, then $\bar{\delta}$ at that point is taken as the approximate solution to \eqref{optimize_J(delta)}. Although we omit further details, as can be seen from Algorithm \ref{alg_goldensect}, the computation of $\mathcal{J}(\bar{m}_k^{\delta}, \bar{w}_k^{\delta})$ is required four times at each iteration. Since the computation of $\mathcal{J}(\bar{m}_k^{\delta}, \bar{w}_k^{\delta})$ requires some numerical integrations, the computational cost increases significantly as we make the tolerance $\kappa$ smaller. In order to verify this, we examine two cases of the tolerance: $\kappa = 10^{-5}$ and $\kappa = 10^{-15}$. The exploitability-based step-size (S3) is straightforward. 
Recall that (S1), (S2), and (S3) are called the adaptive step-sizes. 
For the predefined step-size (S4), we consider the case where $\delta_k = k_0/(k+k_0)$, where $k_0 \ge 1$ is a constant. As mentioned earlier, the case $k_0 = 1$ implies the fictitious-play iterative method $\left(\bar{m}_k = \sum_{j=0}^{k-1} (1/k)m_j\right)$. 

\begin{algorithm}[t]
 \caption{Golden-section method}
 \begin{algorithmic}[1]
 \State Choose a tolerance $\kappa \ll 1$. Set $a = 0, d = 1$ and $\vhi = (1+\sqrt{5})/2$. 
 \While {$d-a > \kappa$}
 \State Set $(b, c) = \left(d-(d-a)/\vhi, a+(d-a)/\vhi \right)$ and find
$$
\bar{\delta} := \argmin{\delta \in \{a, b, c, d\}}{\mathcal{J}(\bar{m}_k^{\delta}, \bar{w}_k^{\delta})}.
$$
 \If {$\bar{\delta} = a$ ($b, c, d$)}
 \State Set $d=b$ ($d=c, a = b, a = c$). 
 \EndIf
 \EndWhile
\end{algorithmic}
\label{alg_goldensect}
\end{algorithm}

\subsection{Computation time and number of iterations}
\label{sec_numerical_experiment_2}

\begin{example}
\label{ex:1}
Considering the case $d=2$, we set the parameters as follows:
\begin{align*}
\text{$T, \nu$ and Hamiltonian}: & \qquad T = 0.25, \qquad \nu = 0.01, \qquad H(t, x, y, p) = \frac{1}{2}|p|^2,\\
\text{terminal condition}: & \qquad g(x, y) = - \frac{1}{4 \pi} \left[\cos{(2 \pi x)} + \cos{(2 \pi y)}\right],\\
\text{initial condition}: & \qquad m_0(x, y) = \frac{1}{2 \pi \sigma^2} \exp{\left(-\frac{\left|\binom{x}{y} - \binom{1/2}{1/2}\right|^2}{2 \sigma^2}\right)} \quad (\sigma = 0.2),\\
\text{coupling term}: & \qquad f(x, y, m) = \left|\binom{x}{y} - \binom{1/2}{1/2}\right|^2 + 2 \min{\{m(x, y), 5\}}.
\end{align*}
At time $t = 0$, the density distribution follows the normal distribution with mean $(1/2, 1/2)$ and variance $\sigma^2$. 
\add{We note that the variance $\sigma^2$ is chosen sufficiently small so that $m_0$ is numerically negligible near the boundary of the computational domain. Hence, its periodic extension on $\mathbb{T}^d$ provides an accurate representation of the Gaussian initial density used in the numerical experiments.} 
At time $t = T$, the control input is $\displaystyle{-\nabla_p H (T, x, y, \nabla g) = - \frac{1}{2} \binom{\sin{(2 \pi x)}}{\sin{(2 \pi y)}}}$; thus the density distribution evolves in such a way that it disperses toward the four corner of $[0, 1]^2$. Regarding the coupling term, the first term $\left|(x, y) - (1/2, 1/2)\right|^2$ plays a role in concentrating the density distribution at $(1/2, 1/2)$. However, under this setting, this effect is not clearly visible because the density distribution is already concentrated at $(1/2, 1/2)$ at time $t = 0$. Also, the second term $\min{\{m(x, y), 5\}}$ plays a role in mitigating the congestion. Note that we impose the upper bound to guarantee the uniform boundedness of $f$.
\end{example}

In the setting of Example \ref{ex:1}, the evolution of the density distribution is shown in Figure \ref{figure1}. Also, Figure \ref{figure2} shows the control input. (Since the control remains almost unchanged due to the short time scale, we plot it only at time $t = T$.)

\begin{figure}[ht]
\centering
{\includegraphics[height=8cm]{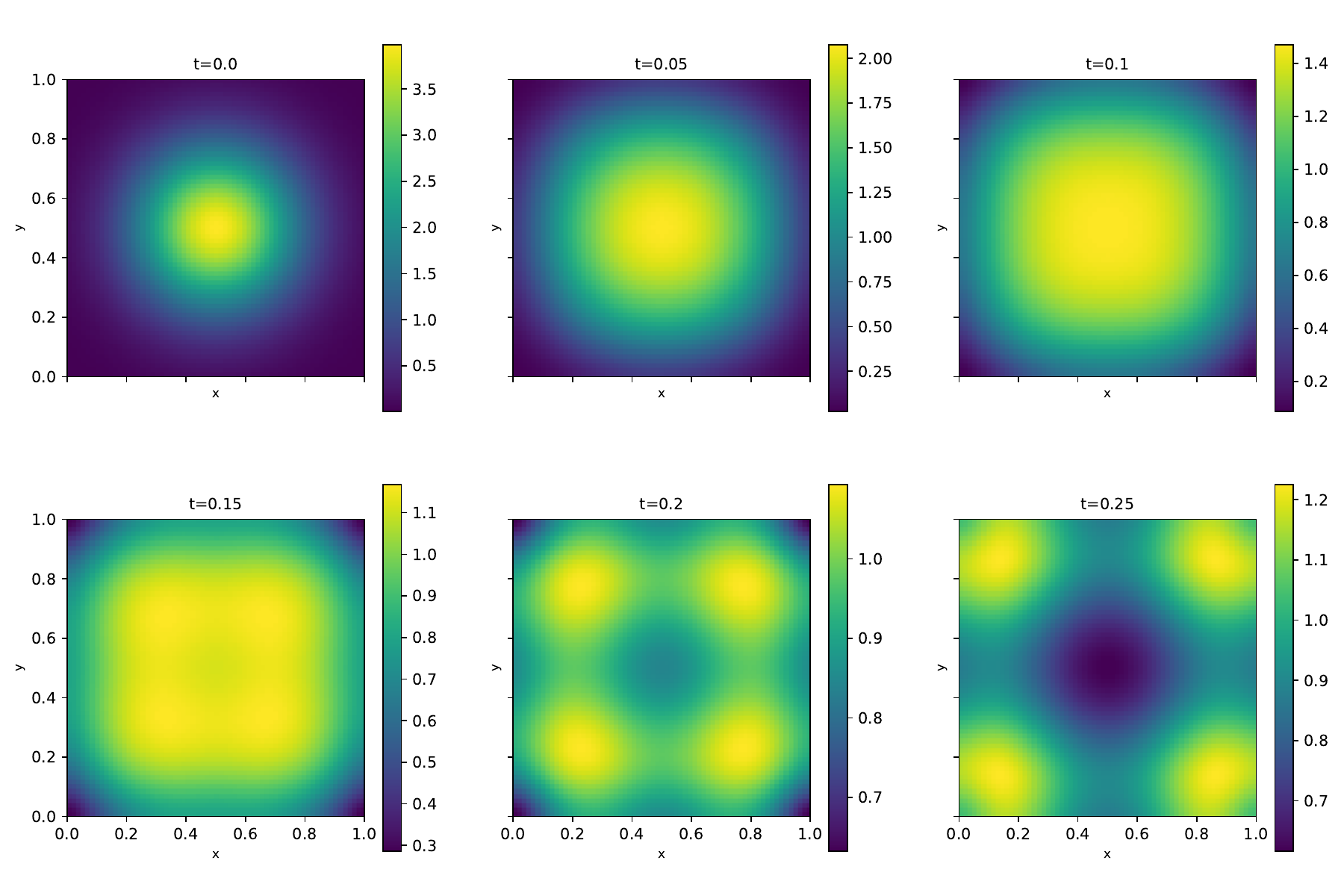}}
\caption{Density distribution $m$}
\label{figure1}
\end{figure}

\begin{figure}[ht]
\centering
{\includegraphics[height=5cm]{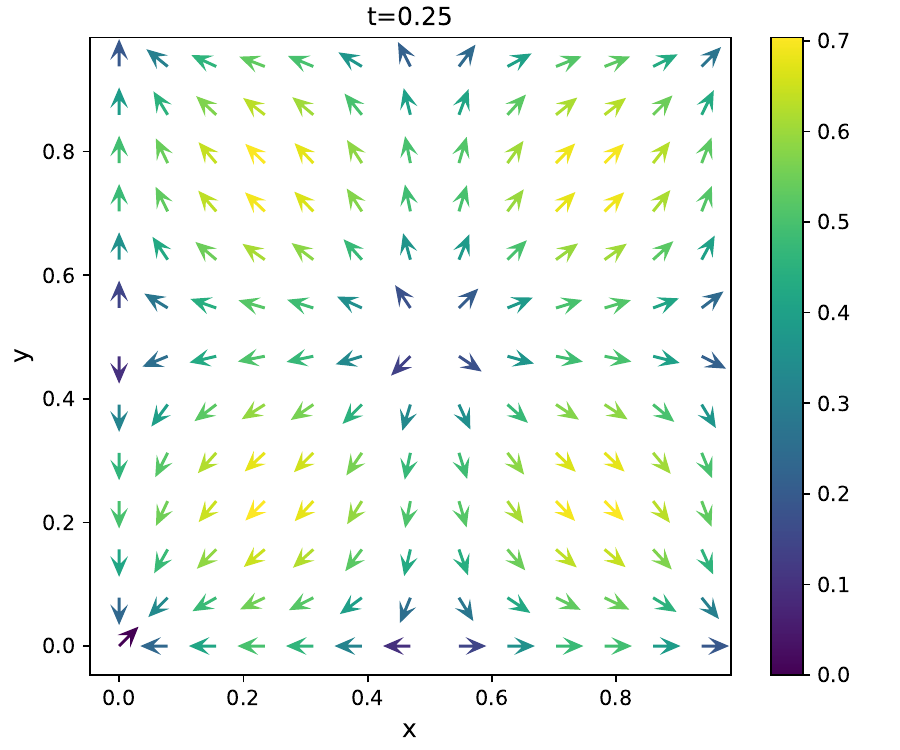}}
\caption{Vector field of the control input at time $t = T$ (the scale represents the norm)}
\label{figure2}
\end{figure}

We compare the computation time and the number of iterations needed for convergence.
We set $N_t = N_x = 40$. Table \ref{table1} shows the computation time and the number of iterations required to achieve $\sigma_k < 10^{-5}$ or the number of iteration reaches $1000$.  (All computations were performed in Python on a laptop equipped with an AMD Ryzen AI 9 365 processor (10 cores) and 32 GB of RAM.) However, since the non-negativity of the numerical approximation of $\sigma_k$ is not guaranteed, the iteration is terminated as soon as $\sigma_k < 0$ holds.

\begin{table}[t]
\caption{Computation time and number of iterations}
\centering
\label{table1}
\begin{tabular}{|c|c|r|r|} \hline
 \multicolumn{4}{|c|}{Adaptive stepsizes} \\ \hline
 \multicolumn{2}{|c|}{$\delta_k$} & computation time [s] & number of iterations \\ \hline
 \multicolumn{2}{|c|}{Optimal ($\kappa = 10^{-5}$)} & 56.0194 & 63\\ \hline
 \multicolumn{2}{|c|}{Optimal ($\kappa = 10^{-15}$)} & 111.2467 & 63\\ \hline
 \multicolumn{2}{|c|}{QAG} & 30.6666 & 78\\ \hline
 \multicolumn{2}{|c|}{Exploitability-based} & 25.0248 & 73\\ \hline
 \end{tabular}
\begin{tabular}{|c|c|r|r|} \hline
 \multicolumn{4}{|c|}{predefined stepsizes} \\ \hline
 \multicolumn{2}{|c|}{$\delta_k$} & computation time [s] & number of iterations \\ \hline
 \multirow{4}{*}{$\delta_k = \dfrac{k_0}{k+k_0}$} & $k_0=1$ & 342.9919 & 1000\\ \cline{2-4}
 & $k_0=5$ & 38.8044 & 113\\ \cline{2-4}
 & $k_0=10$ & 32.1437 & 94\\ \cline{2-4}
 & $k_0=100$ & 116.6502 & 341\\ \hline
 \end{tabular}
\end{table}

From Table \ref{table1}, we observe the expected result that adaptive step sizes require fewer iterations to achieve the criterion $\sigma_k<10^{-5}$ compared with predefined step-sizes. We also note that the optimal step-sizes with $\kappa=10^{-15}$ yield the smallest number of iterations, although the corresponding computation time becomes significantly longer.

By examining the behavior of the exploitability and the step-sizes (left panels of Figures \ref{figure3} and \ref{figure4}), we see that the two cases $\kappa=10^{-15}$ and $\kappa=10^{-5}$ behave almost identically up to around the 60th iteration.
When the iteration is continued beyond the criterion, numerical errors start to accumulate, and eventually the iteration progresses further in the case $\kappa=10^{-5}$ than in the case $\kappa=10^{-15}$.
However, this should be regarded as a coincidental success of the iteration rather than a robust phenomenon. In general, a smaller tolerance $\kappa$ is expected to lead to more stable computations.

In the case of predefined step sizes, the convergence tends to become faster as $k_0$ increases.
However, when $k_0$ becomes extremely large, the convergence suddenly deteriorates (see also Remark \ref{rem_thm8-2}).
To investigate this behavior in more detail, we examine the decay of the exploitability.
The right panel of Figure \ref{figure3} shows that the larger $k_0$ becomes, the slower the decay is in the early iterations.
When $k_0$ is sufficiently large, we have $\delta_k\approx 1$ for small $k$. If $\delta_k$ remains close to $1$, the iteration does not proceed successfully.

\begin{figure}[ht]
\begin{minipage}{0.5\hsize}
\centering
{\includegraphics[height=4.5cm]{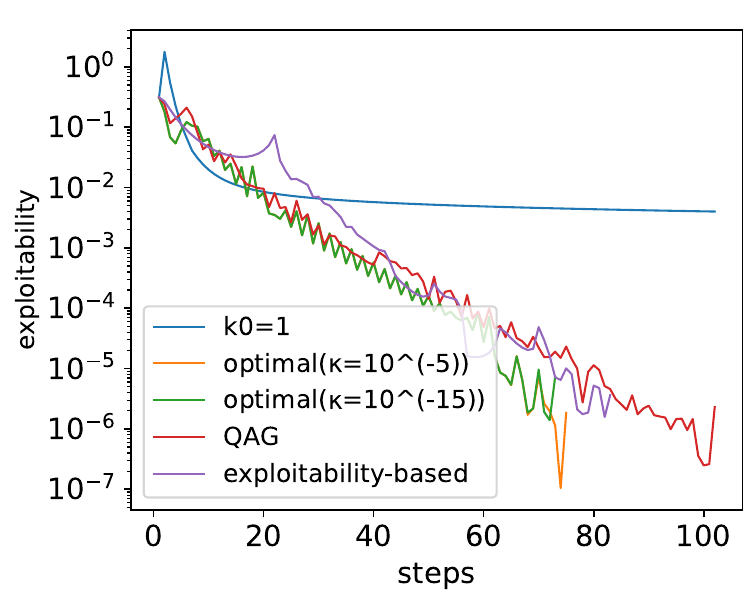}}
\end{minipage}
\begin{minipage}{0.5\hsize}
\centering
{\includegraphics[height=4.5cm]{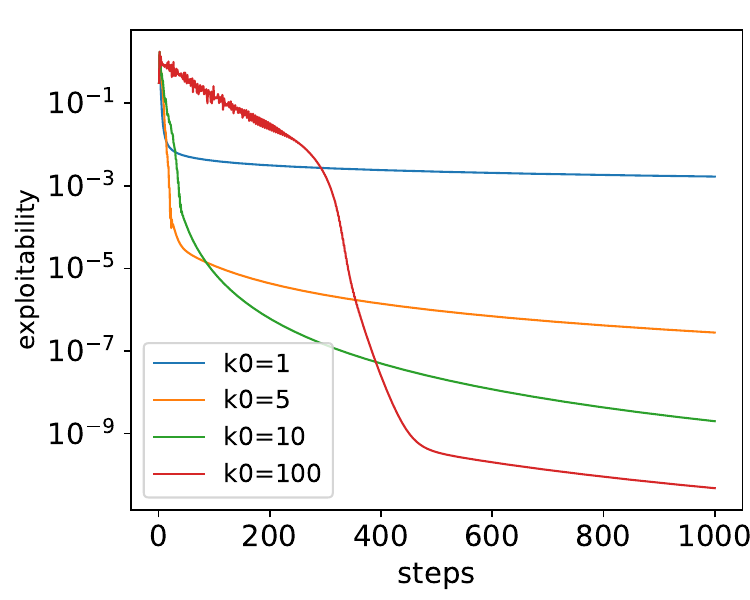}}
\end{minipage}
\caption{Adaptive stepsizes (left figure) and predefined stepsizes (right figure)}
\label{figure3}
\end{figure}

\begin{figure}[ht]
\begin{minipage}{0.5\hsize}
\centering
{\includegraphics[height=4.5cm]{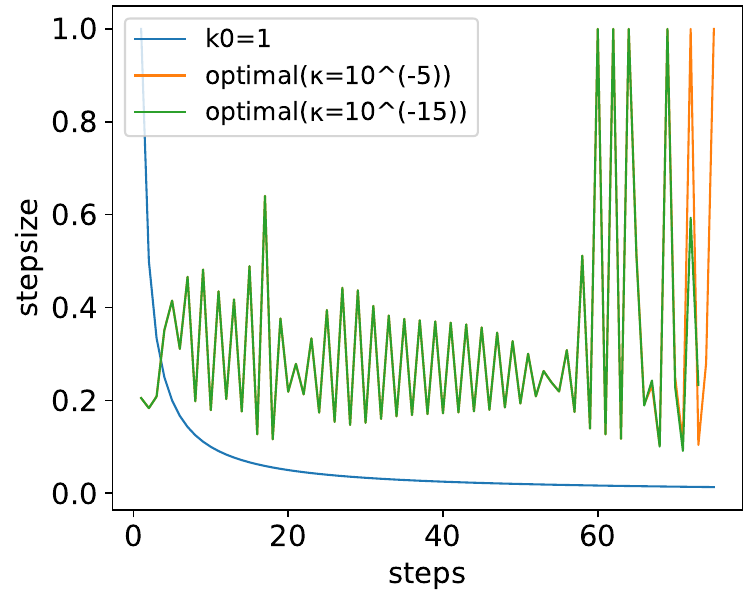}}
\end{minipage}
\begin{minipage}{0.5\hsize}
\centering
{\includegraphics[height=4.5cm]{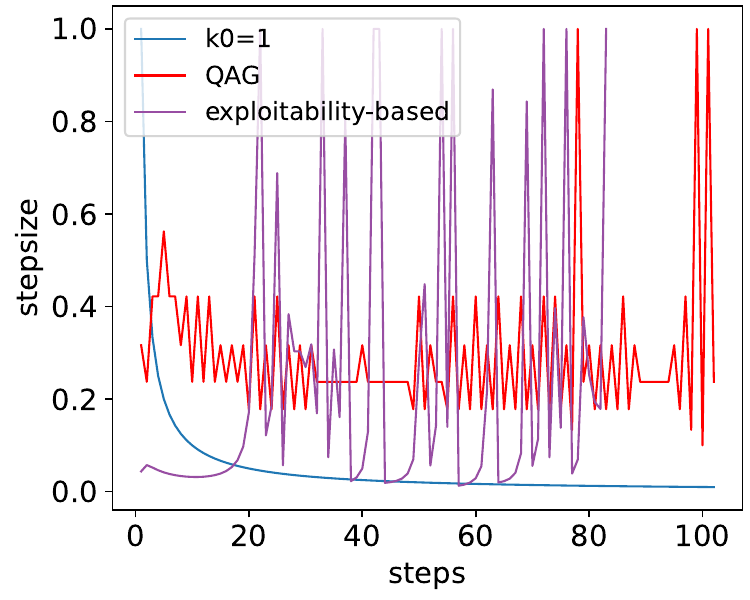}}
\end{minipage}
\caption{Step-sizes $\delta_k$}
\label{figure4}
\end{figure}

\subsection{Convergence rates of iterartions}
\label{sec_numerical_experiment_3}

We present numerical experiments to verify the explicit convergence rates of the GCG method stated in Theorems \ref{LP23_thm7} and \ref{LP23_thm8-1} and \ref{LP23_thm8-2}. However, there is a significant issue. While we have mathematically established convergence rate estimates with respect to the number of iterations, we have not studied the discretization error at all.
In fact, we have succeeded in deriving explicit estimates for both the rate of convergence of iterations and the discretization error of the GCG method discretized by the finite difference method. This result will be reported in a forthcoming paper \cite{NS26}. 
In this paper, we restrict ourselves to a one-dimensional spatial setting and employ sufficiently fine meshes so that the influence of spatial and temporal discretization can be neglected. We thus focus solely on investigating the convergence speed of the GCG iteration through numerical experiments.

If $\eps_k$ is sufficiently small, we see from Theorem \ref{LP23_thm7} that 
$$
\frac{\|(\bar{m}_k, \bar{w}_k) - (\bar{m}, \bar{w})\|_{\ast}}{\sqrt{\eps_k}} \approx C,
$$
where 
$$
\|(\bar{m}_k, \bar{w}_k) - (\bar{m}, \bar{w})\|_{\ast} := \|\bar{m}_{k}-\bar{m}\|_{{L^2(0, T; L^{\infty}(\T^d))}} + \|\bar{w}_{k} - \bar{w}\|_{L^2(Q; \R^d)}
$$
Further, if $\eps_k$ is sufficiently small, Theorem \ref{LP23_thm8-1} suggests the following trend: 
$$
\frac{\eps_{k+1}}{\eps_k} \approx \lambda \in (0, 1).
$$
for the adaptive step-sizes. 

\begin{example}
\label{ex:2}
Considering the case $d=1$, and we set
\begin{align*}
\text{$T, \nu$ and Hamiltonian}: & \qquad T = 0.1, \qquad \nu = 0.01, \qquad H(t, x, p) = \frac{1}{2}|p|^2,\\
\text{terminal condition}: & \qquad g(x) = -\frac{1}{2 \pi} \cos{(2 \pi x)}, \\
\text{initial condition}: & \qquad m_0(x) = \frac{1}{\sqrt{2 \pi \sigma^2}} \exp{\left(-\frac{(x-1/2)^2}{2 \sigma^2}\right)} \quad (\sigma = 0.1),\\
\text{coupling term}: & \qquad f(x, m) = \left(x-\frac{1}{2}\right)^2 + 4 \min{\{m(x), 5\}}.
\end{align*}
This problem is similar to the one obtained by reducing the two-dimensional setting considered previously to a one-dimensional case.
The evolution of the density distribution also behaves similarly to that in the two-dimensional case (see Figure \ref{figure5}). 
\end{example}

\begin{figure}[ht]
\centering
{\includegraphics[height=4cm]{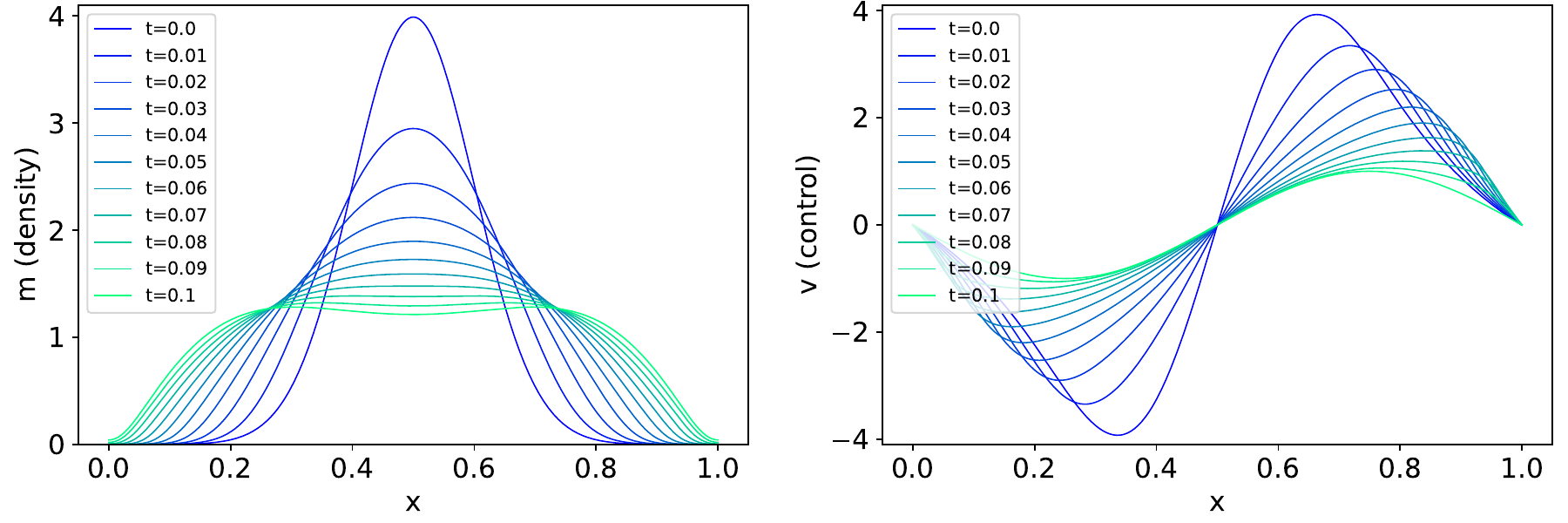}}
\vspace{-8pt}
\caption{Density distribution (left figure) and control input (right figure)}
\label{figure5}
\end{figure}

We set $N_t = 2000, N_x= 500$ and fix the tolerance of the golden-section method to $\kappa = 10^{-15}$. Although the computation of
$$
\|(\bar{m}_k, \bar{w}_k) - (\bar{m}, \bar{w})\|_{\ast}, \qquad \eps_k = \mathcal{J}(\bar{m}_k, \bar{w}_k) - \mathcal{J}(\bar{m}, \bar{w})
$$
requires the exact solution $(\bar{m}, \bar{w})$, it is not available in closed form. Therefore, we use as a reference solution $(\bar{m}, \bar{w})$ the result obtained with the step-size $\delta_k=10/(k+10)$ and a sufficiently large number of iterations.

\begin{figure}[ht]
\begin{minipage}{0.5\hsize}
\centering
{\includegraphics[height=4.5cm]{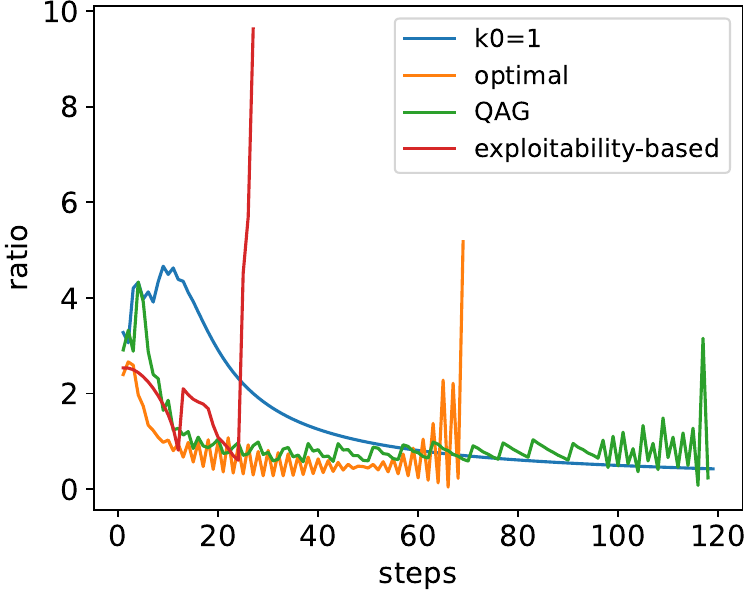}}
\caption{$\|(\bar{M}_k, \bar{W}_k) - (\bar{m}, \bar{w})\|_{\ast}/\sqrt{\eps_k}$}
\label{figure6}
\end{minipage}
\begin{minipage}{0.5\hsize}
\centering
{\includegraphics[height=4.5cm]{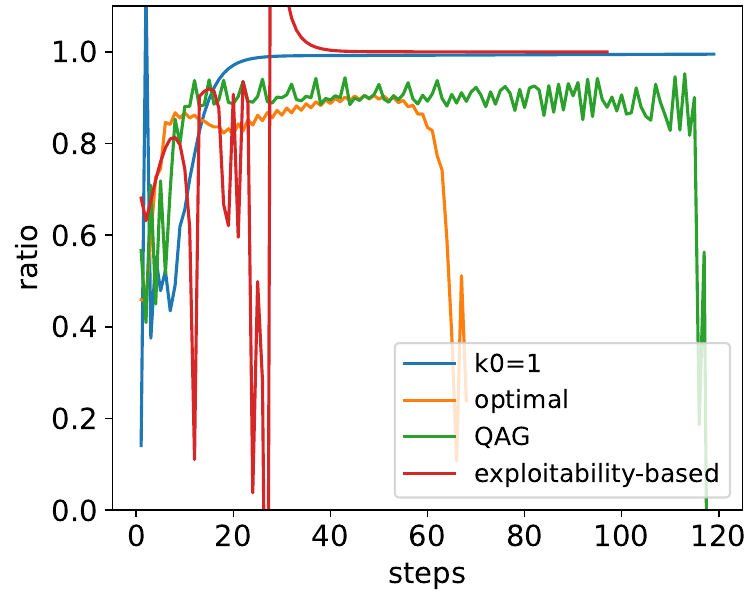}}
\caption{$\eps_{k+1}/\eps_k$}
\label{figure7}
\end{minipage}
\end{figure}

As shown in Figures \ref{figure6} and \ref{figure7}, the computation proceeds successfully for the predefined step-size $\delta_k = 1/(k+1)$, whereas it becomes unstable for the adaptive step-sizes due to inaccuracies in the numerical integration.
Nevertheless, as long as the numerical accuracy is maintained, the results behave roughly as expected.
In particular, Figure \ref{figure7} indicates that $\lambda \approx 0.9$ for both the optimal step-size rule and the QAG condition.

On the other hand, the computation appears particularly unstable when using exploitability-based step-sizes.
From Figures \ref{figure8} and \ref{figure9}, we observe that the exploitability continues to decrease, while the error stagnates after a certain point.
This suggests that the numerical solution converges to a state that differs from the reference solution.
Empirically, only in the case of exploitability-based step-sizes does the iteration sometimes proceed even though the computation is unsuccessful.
(For the other adaptive rules, the iteration stops immediately when the problem becomes ill-conditioned.)

The evolution of the step-sizes in Figure \ref{figure10} shows that they converge to zero.
Figure \ref{figure11} illustrates the decay of $\sigma_k$ and $D_k$, where $D_k$ is defined by \eqref{eq:s3-dk}. 
We observe that $D_k$ stops decreasing after some point, meaning that the approximation of $\sigma_k$ becomes significantly smaller than that of $D_k$. 
Consequently, the guess is hardly updated, and the iteration converges to an incorrect solution.

\begin{figure}[ht]
\begin{minipage}{0.5\hsize}
\centering
{\includegraphics[height=4.5cm]{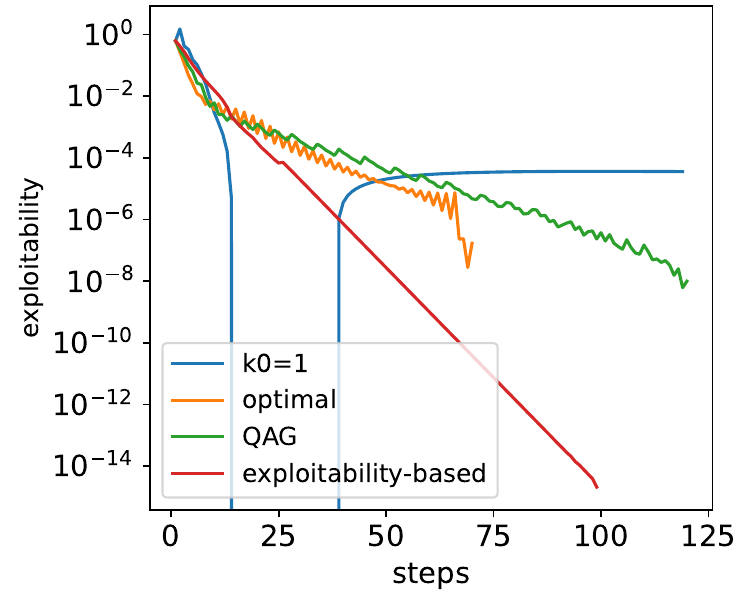}}
\caption{Exploitability $\sigma_k$}
\label{figure8}
\end{minipage}
\begin{minipage}{0.5\hsize}
\centering
{\includegraphics[height=4.5cm]{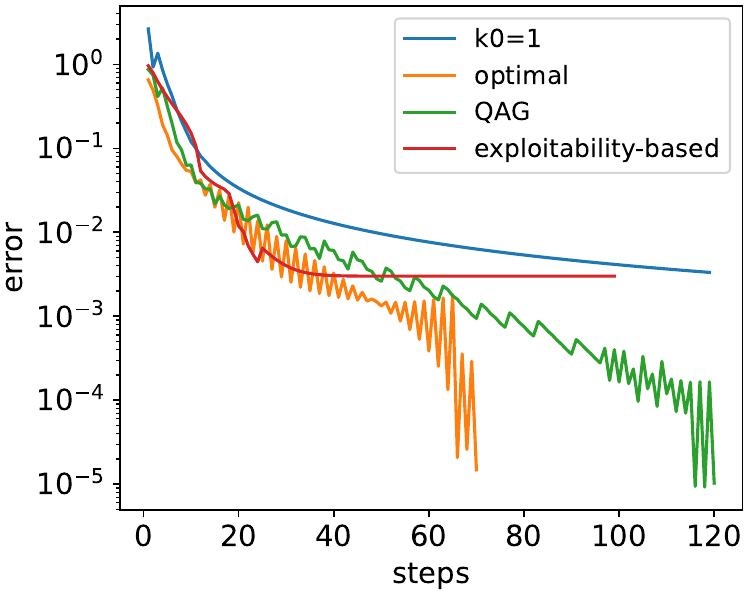}}
\caption{Error $\|(\bar{M}_k, \bar{W}_k) - (\bar{m}, \bar{w})\|_{\ast}$}
\label{figure9}
\end{minipage}
\end{figure}

\begin{figure}[ht]
\begin{minipage}{0.5\hsize}
\centering
\centering
{\includegraphics[height=4.5cm]{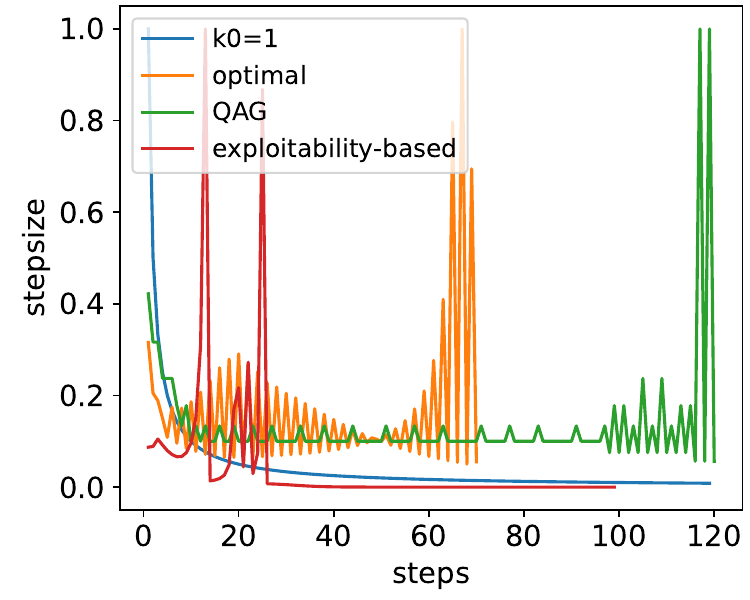}}
\caption{Step-sizes $\delta_k$}
\label{figure10}
\end{minipage}
\begin{minipage}{0.5\hsize}
\centering
{\includegraphics[height=4.5cm]{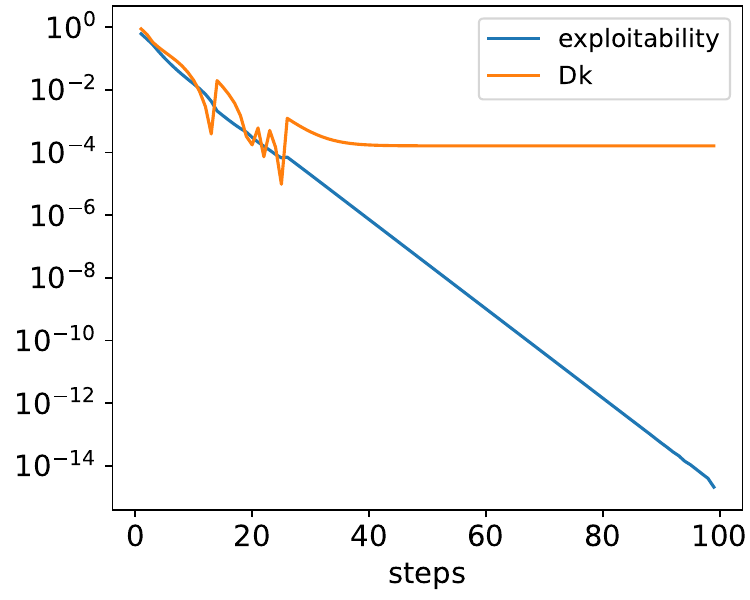}}
\caption{Exploitability $\sigma_k$ and $D_k$}
\label{figure11}
\end{minipage}
\end{figure}

\section{Concluding remarks}
\label{sec:cr}

\begin{enumerate}
\item \emph{Summary.} We have extended the convergence analysis of the GCG method to the MFG system with locally coupled interaction terms. The previous results in the literature were restricted to globally coupled terms, whereas the proposed framework accommodates congestion-type costs commonly used in numerical and practical models. We established explicit convergence rates for both adaptive and predefined step-size selections, although the decay becomes dimension-dependent due to Sobolev embedding constraints. Theoretical findings were validated through numerical experiments, and we additionally proved existence and uniqueness of smooth solutions under the same assumptions. 
    
\item \emph{Extension to general Hamiltonians.}
The major difficulty in treating local couplings stems from the necessity of uniform Lipschitz estimates for the value function $u$. To guarantee this, we restricted ourselves to a quadratic Hamiltonian with a convection effect. Removing this restriction is an important direction for future research. In particular, general smooth and convex Hamiltonians considered in prior GCG analyses remain out of reach because the Cole–Hopf transformation no longer applies. \add{At the same time, one needs to devise a method that does not rely on the strong convexity of $L$.} Developing new regularity arguments to ensure a priori gradient bounds for the Hamilton--Jacobi--Bellman will be essential for handling the general case and broadening the applicability of the theory.

\item \emph{Fully discrete schemes and error estimates.}
In this work, numerical experiments relied on a finite difference method of \cite{I23}, while discretization error analysis was deliberately excluded. An important next step is to construct a fully discrete GCG scheme within a systematic framework of approximations such as finite difference or finite element methods, and to derive corresponding error estimates. 
In particular, explicitly quantifying the overall convergence rate in terms of both the discretization error and the iteration error will be a key contribution for developing practically reliable solvers.
Progress on this topic will be reported elsewhere, including \cite{NS26}.

\item \emph{Remarks on the terminal cost and second coupling term.} 
In this paper, we assume that the terminal condition \add{$g=g(x)$} is prescribed and do not include a terminal cost \add{$g=g(x,m(T,\cdot))$}. 
Although considering a terminal cost is often preferable from the viewpoint of optimal control, we avoid doing so here in order to keep the presentation from becoming overly technical.
Similarly, we do not incorporate the effect of the price function, i.e., the second coupling term, into our analysis.
However, we remark that the arguments developed in this paper can be adapted to the problem setting that includes these features without substantial difficulty.

\item \add{\emph{Study on the backtracking line-search strategy.} 
The backtracking line-search strategy proposed in \cite[Algorithm 2]{YCLZ24} as an adaptive choice of the step size $\delta_k$ is particularly interesting. Since it exploits geometric properties of the best response, it may incorporate information that is overlooked by other step-size selection strategies. Obtaining explicit convergence rates for the iterative sequence generated by this method, as well as performing numerical comparisons with the other adaptive step-size selection rules considered in the present paper and in \cite{LP23}, constitute important directions for future research.}

\end{enumerate}

%%Acknowledgements
\section*{Acknowledgements}
We would like to thank Dr. Daisuke Inoue (Toyota Central R{\&}D Labs., Inc.) and Professor Takahito Kashiwabara (The University of Tokyo)
for their valuable advice and insightful discussions during the course of this research. \add{We also thank the anonymous reviewers for their valuable comments and suggestions to improve the quality of the paper.}
This work was partially supported by JSPS KAKENHI Grant Number 21H04431 \add{and 26H02001} (Grant-in-Aid for Scientific Research (A)). Nakamura was supported by the SPRING GX program of The University of Tokyo.

%%%
%%% References 
%%%
\bibliographystyle{plain}
\bibliography{Bibliography}

\section*{Statements \& Declarations}

\noindent \textbf{Funding:} This work was supported by JST SPRING (Grant Number JPMJSP2108) and partially supported by JSPS KAKENHI (Grant Number 21H04431, \add{26H02001}, Grant-in-Aid for Scientific Research (A)).

\medskip

\noindent \textbf{Competing Interests:} The authors have no relevant financial or non-financial interests to disclose.

\medskip

\noindent \textbf{Author Contributions:} All authors contributed to the conception and design of the study. Haruka Nakamura wrote the first draft of the manuscript and Norikazu Saito contributed to subsequent revisions. All authors read and approved the final manuscript.

\end{document}